\documentclass[12pt]{article}
\usepackage{amssymb,amsmath, amsthm}
\usepackage{color,xcolor}
\newtheorem{lemma}{Lemma}[section]
\newtheorem{proposition}[lemma]{Proposition}
\newtheorem{theorem}[lemma]{Theorem}

\newtheorem{corollary}[lemma]{Corollary}
\newtheorem{definition}[lemma]{Definition}

\newcommand{\im}{\mathrm{Im}}
\newcommand{\Kahler}{K\"ahler\,\,}
\newcommand{\eqand}{\textrm{\,\,\,\,\,\,\,\,\,\,\,\,and\,\,\,\,\,\,\,\,\,\,\,\,}}

\newcommand{\vs}{\vskip 0.2cm}

\allowdisplaybreaks

\makeatletter
\renewcommand*\env@matrix[1][*\c@MaxMatrixCols c]{%
  \hskip -\arraycolsep
  \let\@ifnextchar\new@ifnextchar
  \array{#1}}
\makeatother

\begin{document}

\title{Multiplicities of the Betti map associated to a section of an elliptic surface from a differential-geometric perspective}

\author{Ngaiming Mok\footnote{Department of Mathematics, The University of Hong Kong, Hong Kong, People's Republic of China. \textbf{Email:}~nmok@hku.hk},\,\, Sui-Chung Ng\footnote{School of Mathematical Sciences, Shanghai Key Laboratory of PMMP, East China Normal University, Shanghai, People's Republic of China. \textbf{Email:}~scng@math.ecnu.edu.cn}}

\date{}

\maketitle

\noindent
\begin{abstract}

\noindent
For the study of the Mordell-Weil group of an elliptic curve ${\bf E}$ over a complex function field of a projective curve $B$, the first author introduced the use of differential-geometric methods arising from K\"ahler metrics on $\mathcal H \times \mathbb C$ invariant under the action of the semi-direct product ${\rm SL}(2,\mathbb R) \ltimes \mathbb R^2$. To a properly chosen geometric model $\pi: \mathcal E \to B$ of ${\bf E}$ as an elliptic surface and a non-torsion holomorphic section $\sigma: B \to \mathcal E$ there is an associated ``verticality'' $\eta_\sigma$ of $\sigma$ related to the locally defined Betti map.  The first-order linear differential equation satisfied by $\eta_\sigma$, expressed in terms of invariant metrics, is made use of to count the zeros of $\eta_\sigma$, in the case when the regular locus $B^0\subset B$ of $\pi: \mathcal E \to B$ admits a classifying map $f_0$ into a modular curve for elliptic curves with level-$k$ structure, $k \ge 3$, explicitly and linearly in terms of the degree of the ramification divisor $R_{f_0}$ of the classifying map, and the degree of the log-canonical line bundle of $B^0$ in $B$. Our method highlights ${\rm deg}(R_{f_0})$ in the estimates, and recovers the effective estimate obtained by a different method of Ulmer-Urz\'ua on the multiplicities of the Betti map associated to a non-torsion section, noting that the finiteness of zeros of $\eta_\sigma$ was due to Corvaja-Demeio-Masser-Zannier. The role of $R_{f_0}$ is natural in the subject given that in the case of an elliptic modular surface there is no non-torsion section by a theorem of Shioda, for which a differential-geometric proof had been given by the first author.  Our approach sheds light on the study of non-torsion sections of certain abelian schemes.

\end{abstract}

\section{Introduction}

For the study of a polarized family of abelian varieties over a quasi-projective manifold $\pi_0: \mathcal A^0 \to X^0$, together with a nonsingular projective compactification $\pi: \mathcal A \to X$, one of the authors introduced in~\cite{Mo} the use of methods of complex differential geometry by making use of K\"ahler metrics first introduced by Satake~\cite{Sa} on $\mathcal H_g \times \mathbb C^g$, for the Siegel upper half plane $\mathcal H_g$ of genus $g$, invariant under the standard action of the semidirect product Sp$(g,\mathbb R) \ltimes \mathbb R^{2g}$.  For the case of $g = 1$, let $\Gamma = \Gamma(k) \subset {\rm SL}(2,\mathbb Z)$ be a principal congruence subgroup for some $k \ge 3, X_\Gamma^0 := \mathcal H/\Gamma, \ {\rm and} \ X_\Gamma^0 \subset X_\Gamma$ be a smooth projective compactification of $X_\Gamma$. In~\cite{Mo} a differential-geometric proof of the theorem of Shioda~\cite{Sho} was given yielding the finiteness of the Mordell-Weil group of the universal elliptic curve ${\bf E}$ over the function field $\mathbb C(X_\Gamma)$.  For the study of the universal elliptic curve one makes use of a geometric model $\pi: \mathcal E_\Gamma \to X_\Gamma$, an elliptic modular surface, and realizes the abelian group of rational points of ${\bf E}$ over $\mathbb C(X_\Gamma)$ as the group of holomorphic sections of $\mathcal E_\Gamma$ over $X_\Gamma$. In~\cite{Mo}, to any such section $\sigma$ there was an associated ``vertical'' part $\eta_\sigma$ of $d\sigma$ which was proven to be square-integrable with respect to a natural Hermitian metric $\theta$ on $K_{X^0_\Gamma}\otimes V$, where $V$ is a square root of the holomorphic tangent bundle $T(X^0_\Gamma)$.  The tensor $\eta_\sigma$ was further proven to satisfy an eigensection equation $\overline{\nabla}^*\overline{\nabla}\eta_\sigma = -\eta_\sigma$ for covariant differentiation $\overline{\nabla}$ in the $(0,1)$-direction on $(K_{X^0_\Gamma}\otimes V,\theta)$, from which it was concluded that $\eta_\sigma = 0$, so that $\sigma$ is a torsion section, hence the finiteness of the Mordell-Weil group ${\bf E}(\mathbb C(X_\Gamma))$.

\vs
We will call $\eta_\sigma$ the verticality of $\sigma$ in the current article. For a locally non-isotrivial elliptic surface $\pi: \mathcal E \to B$ arising from a classifying map $f_0: B^0 \to X^0_\Gamma$, where $B^0 \subset B$ is a dense Zariski open subset, it was proven also in~\cite{Mo} that rank(${\bf E}(\mathbb C(B)))$ is bounded by $2\deg(R_{f_0})$, where $\deg(R_{f_0})$ stands for the degree of the ramification divisor $R_{f_0}$ of the classifying map $f_0: B^0 \to X^0_\Gamma$, again by examining the verticalities $\eta_\sigma$ of holomorphic sections of $\pi: \mathcal E \to B$. The authors believe that the invariant K\"ahler metrics of Satake will be applicable to the study of geometric and arithmetic problems concerning abelian schemes over quasi-projective varieties in general, and it was in this context that we found the recent finiteness result of Corvaja-Demeio-Masser-Zannier~\cite{CDMZ} on the number of points of Betti multiplicity $\ge 2$ for a non-torsion holomorphic section $\sigma$ of $\pi: \mathcal E \to B$, a good place to examine the role that can be played by differential-geometric methods.  The verticality $\eta_\sigma$ for a holomorphic section of $\pi: \mathcal E \to B$ turns out to correspond to the differential of the Betti map studied in~\cite{CDMZ}.  On the other hand, by~\cite{Mo} one can derive a {\it holomorphic} tensor $\xi = \nabla \eta_\sigma$ from the verticality $\eta_\sigma$ in view of the identity $\overline{\partial}\nabla\eta_\sigma = 0$ on $B^0$ outside of the ramification divisor $R_{f_0}$, and this gives an effective bound on the sum of multiplicities minus 2 over points at which the Betti map has multiplicities $\ge 3$, linearly in terms of ${\rm deg}(R_{f_0})$ and $\deg(f^*(K_{X_\Gamma}\otimes S_{X_\Gamma})^{\frac{1}{2}}))$, where $f:B\rightarrow X_\Gamma$ is the unique extension of $f_0$ and $S_{X_\Gamma}$ stands for the divisor $X_\Gamma\setminus X^0_\Gamma$, and this gave hope that it might be possible to render the aforementioned finiteness result~\cite{CDMZ} effective in terms of ${\rm deg}(R_{f_0})$ and $\deg(f^*(K_{X_\Gamma}\otimes S_{X_\Gamma})^{\frac{1}{2}}))$. Here the difference is that $\eta_\sigma$ is not holomorphic.  In place of holomorphicity we make use of a first order differential equation satisfied by $\eta_\sigma$, which comes from Mok-To~\cite{MT} (p.52 therein), for the general case of abelian schemes, and which is the source of the elliptic eigensection equations satisfied by $\eta_\sigma$ in the special case of non-isotrivial elliptic surfaces. 

\vs
After the bulk of the work has been finished, we realized that Ulmer-Urz\'ua~\cite{UU} had already obtained an upper bound for the sum in question for any non-torsion section $\sigma$ of an  arbitrary elliptic surface $\pi: \mathcal E \to B$.  The estimate obtained in~\cite{UU} can be deduced from our result on elliptic surfaces with classifying maps, while our result gives structure to the estimates in terms of the ramification divisor in the spirit of~\cite{Mo}.  As our method is quite different from that of~\cite{UU}, we deem it useful to present the differential-geometric approach of the current article, in the belief that the methodology will be applicable to the further research on the Betti map for abelian schemes, e.g., for those in which the classifying map is a generically finite dominant map over some modular variety, such as Shimura curves, Hilbert modular varieties or Siegel modular varieties.

Let $k\geq 3$ be a positive integer. By the canonical action of $ \mathrm{SL}(2,\mathbb Z)\ltimes\mathbb Z^2$ on $\mathcal H\times\mathbb C$, where $\mathcal H$ is the upper half plane, there is the modular family of elliptic curves with level $k$ structure $\pi_k:\mathcal M^0_{\Gamma(k)}:=(\mathcal H\times\mathbb C)/(\Gamma(k)\ltimes\mathbb Z^2)\rightarrow X^0_{\Gamma(k)}:=\mathcal H/\Gamma(k)$, which can be compactified to a relatively minimal elliptic surface $\pi:\mathcal M_{\Gamma(k)}\rightarrow X_{\Gamma(k)}$ over the modular curve $X_{\Gamma(k)}$ using the toroidal compactification~\cite{AMRT}. We will call such $\mathcal M_{\Gamma(k)}$ an \textit{elliptic modular surface}. In this article, we say that an elliptic surface $\pi:\mathcal E\rightarrow B$ over a complex projective curve $B$ \textit{has a classifying map} if there exists a non-constant holomorphic map $f:B\rightarrow X_{\Gamma(k)}$  for some $k\geq 3$, such that the open dense subset $\mathcal E^0\subset\mathcal E$ consisting of the regular fibers is isomorphic to $f|_{B^0}^*\mathcal M_{\Gamma(k)}^0$, where $B^0=\pi(\mathcal E^0)\subset B$. Our main result is an integral formula for the total Betti multiplicity of a given non-torsion section for an elliptic surface with a classifying map:

\begin{theorem}\label{integral formula}
Let $\mathcal E\rightarrow B$ be an elliptic surface over a complex projective curve $B$ with a classifying map $f:B\rightarrow X$ of degree $d$, where $X=X_{\Gamma(k)}$ for some $k\geq 3$. Let $\sigma$ be a non-torsion section of $\mathcal E$ and $m_b$ be the Betti multiplicity of $\sigma$ at $b$, then
$$
\sum_{b\in B}(m_b-1)=\sum_{b\in B\setminus S}(r_b-1)
+\dfrac{d}{2\pi}\int_{X^0}\omega,
$$
where $X^0=X^0_{\Gamma(k)}$ and $S=f^{-1}(X\setminus X^0)$; $r_b$ is the ramification index of $f$ at $b$ and $\omega$ is the K\"ahler form on $X^0$ descending from the invariant form $-i\partial\bar\partial\log\im\tau$ on $\mathcal H$.
\end{theorem}

\noindent\textbf{Remark.} The Betti multiplicities of a section were defined only on the points of good reduction but we will extend the definition for every $b\in B$ (Definition~\ref{betti at cusps}).

When $X=X_{\Gamma(k)}$, $k\geq 3$, we have (\cite{Shm}, Theorem 2.20 therein):
\begin{equation}\label{area formula}
\dfrac{1}{2\pi}\int_{X^0}\omega=g(X)-1+\dfrac{\nu_\infty(X)}{2},
\end{equation}
where $g(X)$ is the genus of $X$ and $\displaystyle\nu_\infty(X)=\dfrac{k^2}{2}\prod_{p|k}\left(1-\dfrac{1}{p^2}\right)\in 2\mathbb N^+$ is the number of cusps on $X$. In other words, the integral in Theorem~\ref{integral formula} is just the degree of the square root of the log canonical bundle $K_{X}\otimes S_{X}$ of $X^0$ in $X$, where $S_{X}$ denotes the divisor (or divisor line bundle) corresponding to the cusps of $X$. Thus,
$$
\sum_{b\in B}(m_b-1)=\sum_{b\in B\setminus S}(r_b-1)+\deg(f^*(K_{X}\otimes S_{X})^{\frac{1}{2}}))=\deg(R_{f_0}\otimes f^*(K_{X}\otimes S_{X})^{\frac{1}{2}})),
$$
where $R_{f_0}:= \bigotimes\limits_{b \in B\setminus S}[b]^{r_b-1}$ is the ramification divisor of the classifying map $f_0: B^0 \to X^0$. As $f^*(K_X\otimes S_X)\otimes R_{f_0} \cong K_B\otimes S$, we can now rewrite the formula in Theorem~\ref{integral formula} as
\begin{equation}\label{formula 2}
\sum_{b\in B}(m_b-1)=\dfrac{1}{2}\deg(K_B\otimes S\otimes R_{f_0}).
\end{equation}

For the arithmetic aspect, a quantity of concern is the number of points of good reduction at which the Betti multiplicity is at least 2. Let $\mathfrak B_\sigma\subset B\setminus S$ be the set of points at which the Betti multiplicity is at least 2. The finiteness of $|\mathfrak B_\sigma|$ was proven by Corvaja-Demeio-Masser-Zannier in~\cite{CDMZ}. By Eq.(\ref{formula 2}), we have the following estimate 
$$
|\mathfrak B_\sigma|\leq g-1+\dfrac{|S|}{2}+\dfrac{\deg(R_{f_0})}{2},
$$
where $g$ is the genus of $B$. We remark that we have defined $S$ to be $f^{-1}(X\setminus X_0)$ and in general $S\subset S'$, where $S'$ is the set of points of bad reduction. However, if we assume that $\mathcal E$ is relatively minimal, then by the uniqueness of the relatively minimal model, we deduce that $S=S'$. In any case, we can always take $|S|$ to be the number of singular fibers on $\mathcal E$ in the estimate.

An upper bound of $|\mathfrak B_\sigma|$ in terms of the genus and the degree of a certain line bundle on $B$ has also been obtained by Ulmer-Urz\'ua~\cite{UU}. Using our Theorem~\ref{integral formula}, we will prove another estimate:
\begin{corollary}\label{inequality cor}
Following the setting of Theorem~\ref{integral formula}, then we have
$$
|\mathfrak B_\sigma|\leq 2g-2-\deg(f^*(K_{X}\otimes S_{X})^{\frac{1}{2}}))+|S|,
$$
where $g$ is the genus of $B$.
\end{corollary}

It turns out that our estimate of $|\mathfrak B_\sigma|$ is equivalent to that given by Ulmer-Urz\'ua (see the end of Section~\ref{counting section}) when $\mathcal E$ has a classifying map. For an arbitrary non-isotrivial relatively minimal elliptic surface, the equivalence will follow by a reduction argument given in Section~\ref{reduction section} using a finite cover over the base curve $B$. (Here we remark that the relative minimality of $\mathcal E$ has also been implicitly assumed in~\cite{UU}.)

Besides the integral formula above for the total Betti multiplicity, the invariant form $\mu$ can also be used to give an alternative proof of the equality of the canonical height $\hat h(\sigma)$ of a section $\sigma:B\rightarrow\mathcal E$ and the integral of $\sigma^*(d\beta_1\wedge d\beta_2)$ over $B\setminus S$. This equality has been established in~\cite{CDMZ}. (Our result differs by a sign due to a difference choice of the ordering of the Betti coordinates.) Here our proof consists of proving two things (see Section~\ref{canonical height}): (1) the invariant form $\mu$ is actually equal to $d\beta_1\wedge d\beta_2$ on $\mathcal E^0=\mathcal E|_{B\setminus S}$; (2) the integral of $\sigma^*\mu$ over $B\setminus S$ gives the canonical height $\hat h(\sigma)$. The fact that $d\beta_1\wedge d\beta_2=\mu$ will be furthermore used to demonstrate that on $\mathcal E^0$, the 2-form $d\beta_1\wedge d\beta_2$  is equal to the restriction of a certain $(1,1)$-current $T$ on $\mathcal E$ studied by DeMarco-Mavraki~\cite{DM} for problems in Arithmetic Dynamics, a relation that has been proven in~\cite{CDMZ} by different methods.

Now we briefly describe the layout of the article. In Section~\ref{mu section}, we will first recall the construction of the semi-definite K\"ahler form $\mu$ on an elliptic surface $\mathcal E$ with a classifying map. Its relationship with the Betti coordinates (or the Betti map) will then be discussed and applied to further relate it to the canonical height. Next, in Section~\ref{verticality section}, we will see that the kernel of the $(1,1)$-form $\mu$ gives a smooth distribution on the holomorphic tangent bundle of $\mathcal E^0$ (i.e. the regular part of $\mathcal E$) transversal to the relative tangent bundle, from which we get a well-defined projection onto the relative tangent bundle and this is the basis of our definition and analysis of the verticality $\eta_\sigma$ of a section $\sigma$ of $\mathcal E$. Afterwards, we will show that $\eta_\sigma$ satisfies a first order PDE, which has been obtained previously in~\cite{MT} in a slightly different form. Using this PDE we can derive many important properties of $\eta_\sigma$, including the discreteness of its zero set and its local behavior. In Section~\ref{counting big section}, we will construct a potential function $\Psi$ on a dense open subset of the projectivized tangent bundle $\mathbb P(\mathcal M)$ on an elliptic modular surface $\mathcal M$, such that its pullback by the tautological lifting of a local section will be just the norm squared of its verticality. With the potential function $\Psi$, it is easier to analyze the asymptotic behavior of $\eta_\sigma$ towards the points of bad reduction and also in the more general case with ramified classifying maps. We will then be in the position to prove our integral formula for the Betti multiplicity (Theorem~\ref{integral formula}) and the effective estimate for the number of points with higher Betti multiplicity in Corollary~\ref{inequality cor}. At the end of Section~\ref{counting big section}, we will relate our results to those obtained by Ulmer-Urz\'ua~\cite{UU}.

\section{Betti map and a semi-definite K\"ahler form on an elliptic surface}\label{mu section}

Let $\mathcal H$ be the upper half plane and consider the action of semidirect product $\mathrm{SL}(2,\mathbb R)\ltimes\mathbb R^2$ on $\mathcal H\times\mathbb C$, defined by
$$
(\tau,z)\mapsto \left(\dfrac{a\tau+b}{c\tau+d}, \dfrac{z}{c\tau+d}+\alpha+\beta\tau\right),
$$
where $(\tau, z)\in\mathcal H\times\mathbb C)$, $\begin{pmatrix}a&b\\c&d\end{pmatrix}\in \mathrm{SL}(2,\mathbb R)$ and $(\alpha,\beta)\in\mathbb R^2$. Consider a principal congruence subgroup $\Gamma:=\Gamma(k)\subset \mathrm{SL}(2,\mathbb Z)$ for some $k\geq 3$ and from the discrete subgroup $\Gamma\ltimes\mathbb Z^2$ we get a modular family of elliptic curves $\pi_\Gamma:\mathcal M^0_\Gamma\rightarrow X^0_\Gamma$, where
$\mathcal M^0_\Gamma:=(\mathcal H\times\mathbb C)/(\Gamma\ltimes\mathbb Z^2)$ and $X^0_\Gamma:=\mathcal H/\Gamma$.

There is a semi-definite K\"ahler form $\mu$ on $\mathcal H\times\mathbb C$ which is invariant under the aforementioned action of $\mathrm{SL}(2,\mathbb R)\ltimes\mathbb R^2$, given by
$$\mu=i\partial\bar\partial\dfrac{(\im z)^2}{\im\tau}.$$ The invariance and semi-definiteness of $\mu$ are easy to check and we refer the reader to~\cite{MT} for details. Since $\mu$ is invariant, it descends to a semi-definite K\"ahler form on $\mathcal M^0_\Gamma$, which will be still denoted by $\mu$ when there is no danger of confusion. Here we will explicitly describe the kernel of $\mu$, which will play an important role in our analysis of the sections of elliptic surfaces.

\begin{proposition}\label{kernel}
Fix $(\tau_0,z_0)\in\mathcal H\times\mathbb C$ and write $z_0=a+b\tau_0$, where $a,b\in\mathbb R$. Then the kernel of $\mu$ at $(\tau_0, z_0)$ is spanned by the holomorphic tangent vector of the curve $z=a+b\tau$.
\end{proposition}
\begin{proof}
For $z=a+b\tau$, then $\im z=b\,\im\tau$ and hence the K\"ahler potential restricted to the curve $z=a+b\tau$ is $b^2\im\tau$, which is harmonic. Hence, the holomorphic tangent space of the curve $z=a+b\tau$ is a null space for $\mu$.
\end{proof}

Let $U\subset\mathcal M^0_\Gamma$ be a connected open set. If $U$ is sufficiently small, we can find an open set $\mathfrak U\subset\mathcal H\times\mathbb C$ such that $\pi:\mathfrak U\rightarrow U$ is a biholomorphism. Fix such a choice for $\mathfrak U$. Let $p\in U$ and $\pi^{-1}(p)=(\tau, z)\in\mathfrak U$, then there exist unique $\beta_1,\beta_2\in\mathbb R$ such that $z=\beta_1+\beta_2\tau$. When restricting to the image of a local section of $\mathcal M^0_\Gamma$, the pair $(\beta_1,\beta_2)$ is called the \textit{Betti coordinates} in the literature. Here we regard them as local real analytic functions on $U$ and still call them Betti coordinates.

\begin{proposition}\label{beta de} In terms of the local coordinates $(\tau, z)$, we have
$$
\partial\beta_2=\dfrac{dz-\beta_2d\tau}{2i\im\tau}.
$$
\end{proposition}
\begin{proof}
Since $z=\beta_1+\beta_2\tau$, we have $\beta_2\im\tau=\im z$. By differentiation
$$
\beta_2d\tau+(\tau-\bar\tau)\partial\beta_2=dz
$$
and the result follows.
\end{proof}

Of course, $\beta_1$, $\beta_2$ are only locally defined and more importantly they depend on our choice of $\mathfrak U$ (which is sometime called an \textit{abelian logarithm}), but we have the following:

\begin{proposition}\label{2-form prop}
The 2-form $d\beta_1\wedge d\beta_2$ is independent of the choice of abelian logarithm and globally defined on $\mathcal M^0_\Gamma$. Furthermore, we have $\mu=d\beta_1\wedge d\beta_2$.
\end{proposition}
\begin{proof} 
Since $\mu$ is globally defined on $\mathcal M^0_\Gamma$, we just need to show that locally we have $\mu=d\beta_1\wedge d\beta_2$ for any choices of Betti coordinates. As before, take a sufficiently small open set $U\subset\mathcal M^0_\Gamma$ and choose $\mathfrak U\subset\mathcal H\times\mathbb C$ such that $\pi:\mathfrak U\rightarrow U$ is biholomorphic.

By definition, for $(\tau, z)\in\mathfrak U$, we have $z=\beta_1+\beta_2\tau$, hence $\beta_2=\im z/\im\tau$. On the one hand,
\begin{eqnarray*}
\mu&=&i\partial\bar\partial\dfrac{(\im z)^2}{\im\tau}\\
&=&i\partial\left(\dfrac{2\im z}{\im\tau}\bar\partial \im z - \dfrac{(\im z)^2}{(\im\tau)^2}\bar\partial \im\tau\right)\\
&=&2i\partial\left(\dfrac{\im z}{\im\tau}\right)\wedge\left(\bar\partial \im z-\dfrac{\im z}{\im\tau}\bar\partial \im\tau\right)\\
&=&2i\im\tau\partial\left(\dfrac{\im z}{\im\tau}\right)\wedge\bar\partial\left(\dfrac{\im z}{\im\tau}\right)\\
&=&2i\textrm{Im}(\tau)\partial\beta_2\wedge\bar\partial\beta_2.
\end{eqnarray*}
On the other hand, as $z=\beta_1+\beta_2\tau$ $\Rightarrow$ $\bar\partial\beta_1+\tau\bar\partial\beta_2=0$, we have
\begin{eqnarray*}
d\beta_1\wedge d\beta_2&=&(\partial\beta_1+\bar\partial\beta_1)\wedge(\partial\beta_2+\bar\partial\beta_2)\\
&=&\partial\beta_1\wedge\bar\partial\beta_2+\bar\partial\beta_1\wedge\partial\beta_2\\
&=&-\bar\tau\partial\beta_2\wedge\bar\partial\beta_2-\tau\bar\partial\beta_2\wedge\partial\beta_2\\
&=&2i\textrm{Im}(\tau)\partial\beta_2\wedge\bar\partial\beta_2.
\end{eqnarray*}
Therefore,  $\mu=d\beta_1\wedge d\beta_2$ .
\end{proof}

\subsection{Canonical height and the invariant form}\label{canonical height}

Consider the elliptic curve defined by $y^2=x(x-1)(x-\lambda)$ (in an affine chart) over the function field $\mathbb C(\lambda)$ of $\mathbb P^1$. Let $\pi_L:\mathcal E^0_L\rightarrow\mathbb P^1\setminus\{0,1,\infty\}$ be the associated Legendre scheme.
In~\cite{CDMZ}, the following relation between the Betti map and the canonical height for the sections of an elliptic surface given by pulling back by a branched cover $B\rightarrow\mathbb P^1$  has been proven:

\begin{theorem}[\cite{CDMZ}, Theorem 3.2 and Corollary 3.4 therein]\label{height thm}
Let $r:B\rightarrow\mathbb P^1$ be a finite morphism and $B^0:=B\setminus r^{-1}(\{0,1,\infty\})$. Let $\mathcal E^0_{B^0}:=\mathcal E^0_L\times_{(\pi_L,r)} B$ and $\pi_{B^0}:\mathcal E^0_{B^0}\rightarrow B^0$ be the projection. Then, for an algebraic section $\sigma:B^0\rightarrow\mathcal E^0_{B^0}$,
$$
\hat h(\sigma)=\int_{B^0}\sigma^*(d\beta_1\wedge d\beta_2),
$$
where $\hat h$ is the canonical height of $\mathcal E_B$. In particular, the integral above is a rational value.
\end{theorem}

\noindent\textbf{Remark.} In the published version of~\cite{CDMZ}, the integrand in the theorem above is $\sigma^*(d\beta_2\wedge d\beta_1)$, differs by a factor of $-1$. This comes from whether one takes $(1,\tau)$ or $(\tau,1)$ to be the ordered basis for the lattice.

By Proposition~\ref{2-form prop}, we see that the canonical height of $\sigma$ is also equal to the integral of $\mu$ on the image of $\sigma$. We are going to give a direct proof of this for elliptic surfaces given by a classifying maps. The general case can then follow by a reduction procedure similar to that in Section~\ref{reduction section}.

\begin{lemma}
Let $n\in\mathbb N^+$ and define $[n]:\mathcal H\times\mathbb C\rightarrow\mathcal H\times\mathbb C$ by $[n](\tau,z)=(\tau,nz)$. Then $[n]^*\mu=n^2\mu$. 
\end{lemma}
\begin{proof}
The lemma follows directly from the definition $\mu=i\partial\bar\partial\dfrac{(\im z)^2}{\im\tau}$.
\end{proof}

\begin{theorem}
Let $\pi:\mathcal E\rightarrow B$ be an elliptic surface obtain given by some classifying map $f:B\rightarrow X_\Gamma$ for a modular curve $X_\Gamma$ and $\sigma:B\rightarrow\mathcal E$ be a section. Then,
$$
\hat h(\sigma)=\int_{B\setminus S}\sigma^*\mu,
$$
where $S:=f^{-1}(X_\Gamma\setminus X^0_\Gamma)$.
\end{theorem}
\begin{proof}
By the uniqueness of the canonical height (\cite{Si2}, p.218, p.250), it suffices to verify two things:

(1) $\displaystyle\int_{B\setminus S}\sigma^*\mu - \langle\sigma,O\rangle$ is bounded, where $\langle\cdot,\cdot\rangle$ denotes the intersection pairing on $\mathcal E$ and $O$ is the zero section; and

(2) $\displaystyle\int_{B\setminus S}(n\sigma)^*\mu=n^2\int_{B\setminus S}\sigma^*\mu$.

(Note that with an abuse of notation, we use $\sigma$ to both denote the section as a map and its graph.)

We first verify (2), which is easy. Let $f:B\rightarrow X_\Gamma$ be the classifying map and $f^\sharp:\mathcal E^0\rightarrow\mathcal M^0_\Gamma$ be the associated map such that $\pi_\Gamma\circ f^\sharp=f\circ\pi$. To avoid confusion, we denote the semi-definite \Kahler forms on $\mathcal E$ and $\mathcal M^0_\Gamma$ by $\mu$ and $\mu_\Gamma$ respectively. We then have
$$
\int_{B\setminus S}\sigma^*\mu
=\int_{B\setminus S}(f^\sharp\circ\sigma)^*\mu_\Gamma.
$$

Therefore, for $n\in\mathbb N$, 
\begin{eqnarray*}
&&\int_{B\setminus S}(n\sigma)^*\mu
=\int_{B\setminus S}(f^\sharp\circ(n\sigma))^*\mu_\Gamma
=\int_{B\setminus S}([n]\circ f^\sharp\circ\sigma)^*\mu_\Gamma\\
&=&\int_{B\setminus S}(f^\sharp\circ\sigma)^*[n]^*\mu_\Gamma
=n^2\int_{B\setminus S}(f^\sharp\circ\sigma)^*\mu_\Gamma
=n^2\int_{B\setminus S}\sigma^*\mu.
\end{eqnarray*}
We have thus proven that $\displaystyle\int_{B\setminus S}\sigma^*\mu$ is quadratic.

For $(1)$, we first note that, while originally only defined on the set of regular fibers $\mathcal E^0$, the $(1,1)$-form $\mu$ extends as a positive $(1,1)$-current $T$ on $\mathcal E$, which implies in particular that the integral is finite. The extension was first established in the case of elliptic modular surfaces by~\cite{MT} (p.38 therein), and the general case follows by pulling back by the classifying map.  Alternatively, the extension is a consequence of Theorem~\ref{DM thm} and the result of DeMarco-Mavraki~\cite{DM} mentioned before the theorem. (The extension can also be done on the compactified elliptic modular surfaces using the toroidal compactification, see~\cite{BKK} for example.) Whenever a current extension of $\mu$ on $\mathcal E$ exists (which is in general not unique), there is the \textit{trivial extension}, i.e. the $(1,1)$-current $T_\mu$ such that $\langle [T_\mu],\gamma\rangle=\int_{\gamma\cap\mathcal E^0}\mu$ for every algebraic curve $\gamma\subset\mathcal E$, where $[T_\mu]\in H_2(\mathcal E,\mathbb R)$ is the Poincar\'e dual of $T_\mu$. In particular, $\langle[T_\mu],\sigma\rangle=\int_{B\setminus S}\sigma^*\mu$ for every section $\sigma$.

We here recall some classical facts about the theta function on $\mathcal H\times\mathbb C$ defined by
$$
\vartheta_{1,1}(\tau,z)=\sum_{n\in\mathbb Z}
\exp\left(\pi i\tau\left(n+\dfrac{1}{2}\right)^2
+2\pi i\left(z+\dfrac{1}{2}\right)\left(n+\dfrac{1}{2}\right)\right).
$$
(For the relevant details, the reader may see~\cite{BKK}.) It is known that $\vartheta^8_{1,1}$ transforms under the action of $\mathrm{SL}(2,\mathbb Z)\ltimes\mathbb Z^2$ of $\mathcal H\times\mathbb C$ in such a way that it defines a holomorphic section of some line bundle $L$, called \textit{Jacobi line bundle}, on $\mathcal M^0_{\Gamma(k)}$. Moreover, the function 
$h(\tau,z):=(\im\tau)^4\exp(-16\pi(\im z)^2/\im\tau)$ descends to a Hermitian metric for $L$~(see~\cite{BKK}, Lemma 2.11 therein). Thus, $$c_1(L)=-2c_1(\pi^*T(X^0_{\Gamma(k)}))+8\mu.$$

The divisor of $L$ is precisely $8$ times the zero section on $\mathcal M^0_{\Gamma(k)}$ since the zero set of $\vartheta_{1,1}$ is $\{(\tau,m+n\tau)\in\mathcal H\times\mathbb C:\tau\in\mathcal H \textrm{\,\,and\,\,} m,n\in\mathbb Z\}$ and its zeros are of order 1.
Using the toroidal compactification, $L$ extends as a $\mathbb Q$-line bundle $L^\sharp$ (which is also called a \textit{Mumford-Lear extension}) on $\mathcal M_{\Gamma(k)}$. The divisor for $L^\sharp$ is $8O+F$, where $O$ is the zero section and $F$ is some fibral divisor supported on the singular fibers (see~\cite{BKK}, Proposition 4.9 therein).
Since the divisor of a line bundle is Poincar\'e dual to the first Chern form, by comparing with the expression of $c_1(L)$ above, we thus deduce that $8[T_\mu]-8O$ is a fibral divisor and hence $\langle[T_\mu]-O,\sigma\rangle$ is bounded when $\sigma$ varies. This completes the proof for elliptic modular surfaces. In the case where the elliptic surface $\mathcal E\rightarrow B$ is given by a classifying map $f:B\rightarrow X_\Gamma$, we can simply pullback the extensions $L^\sharp$, $T$ and obtain the desired results.
\end{proof}

The 2-form $d\beta_1\wedge d\beta_2$ is also related to the study of canonical heights with equidistribution theory. For an elliptic surface $\mathcal E\rightarrow B$ defined over $\mathbb Q$, DeMarco-Mavraki~\cite{DM} showed that there is a positive closed (1,1)-current $T$ on $\mathcal E(\mathbb C)$ such that 

(1) the restriction of $T$ to each regular fiber is the normalized Haar measure; and 

(2) for a non-torsion section $\sigma:B\rightarrow\mathcal E$ and a non-repeating sequence $t_n\in B(\overline{\mathbb Q})$ such that $\hat h_{E_{t_n}}(\sigma(t_n))\rightarrow 0$, the discrete measures
$
\displaystyle\dfrac{1}{\#\textrm{Gal}(\overline{\mathbb Q}/\mathbb Q)t_n}
\sum_{t\in\textrm{Gal}(\overline{\mathbb Q}/\mathbb Q)t_n}\delta_t
$ 
converge weakly on $B(\mathbb C)$ to $\sigma^*T$.

In addition, on $\mathcal E^0\setminus O$, where $O$ is the image of the zero section, it was proven in~\cite{DM} that $T$ is up to a constant multiple equal to $dd^cH_N$, where $H_N$ is the N\'eron local Archimedean height function. It was then proven in~\cite{CDMZ} that $T$ is in fact equal to $d\beta_1\wedge d\beta_2$. Using the semi-definite K\"ahler form $\mu$, we can give a simpler proof of this equality.

\begin{theorem}\label{DM thm}
Let $\mathcal E$ be an elliptic surface obtained through some classifying map into a modular curve.
On $\mathcal E^0$, we have $T=\mu$. Consequently, we also have $T=d\beta_1\wedge d\beta_2$ on $\mathcal E^0$.
\end{theorem}
\begin{proof}
By pulling back through the classifying maps it suffices to prove the equality for an elliptic modular surface $\mathcal M$. For a sufficiently small open set $U\subset\mathcal M^0$, we idenify it with a branch of its elliptic logarithm $\mathfrak U\subset\mathcal H\times\mathbb C$.

By~\cite{Si2} (p.468), for $(\tau,z)\in U\cong\mathfrak U$,
$$
	H_N=-\dfrac{1}{2}\left(\left(\dfrac{\im z}{\im\tau}\right)^2
	-\dfrac{\im z}{\im\tau}+\dfrac{1}{6}\right)
	(-2\pi\im\tau)+\Phi,
$$
for certain pluriharmonic function $\Phi$. Since $\im\tau$ and $\im z$ are harmonic, we have
$$
	dd^cH_N=\dfrac{i}{\pi}\partial\bar\partial H_N
	=i\partial\bar\partial\dfrac{(\im z)^2}{\im\tau}=\mu.
$$
Finally, it is immediate to verify that the integral of $\mu$ on each regular fiber is equal to 1 and hence $T=\mu$. By Proposition~\ref{2-form prop}, we then also have $T=d\beta_1\wedge d\beta_2$.
\end{proof}

\section{Verticality of a section of an elliptic surface}\label{verticality section}

\subsection{Decomposition of the tangent bundles on elliptic surfaces}\label{decomposition section}

Let $T:=T(\mathcal M^0_\Gamma)$ be the (holomorphic) tangent bundle of $\mathcal M^0_\Gamma$ and $V\subset T$ be the relative tangent bundle associated to the canonical projection $\pi_\Gamma:\mathcal M^0_\Gamma\rightarrow X^0_\Gamma$. We will also call $V$ the \textit{vertical part} of $T$. From Proposition~\ref{kernel}, we see that the kernel of $\mu$ is always transversal to $V$ and we call it the \textit{horizontal part} of $T$ and denote it by $H$. Thus, $T=V\oplus H$. Note that $H$ is only a real-analytic complex line subbundle.

Consider the projection map $\Pi_V:T\rightarrow V$ associated to the decomposition described above. As before, for a sufficiently small open set in $U\subset\mathcal M^0_\Gamma$, by identifying $U$ with one of its lifting to the universal cover, we will use $(\tau, z)\in\mathcal H\times\mathbb C$ as local coordinates on $U$.
Let $(\tau,z)\in U\subset\mathcal M^0_\Gamma$. Write $z=\beta_1+\beta_2\tau$, where $(\beta_1,\beta_2)$ are the \textit{Betti coordinates}. 
For $v\in T_{(\tau,z)}$, where $T_{(\tau,z)}$ denotes the holomorphic tangent space at $(\tau, z)$, we write $v=v_\tau\dfrac{\partial}{\partial\tau}+v_z\dfrac{\partial}{\partial z}$, for some $v_\tau,v_z\in\mathbb C$. Then, the vertical part $V_{(\tau,z)}\subset T_{(\tau,z)}$ is spanned by $\dfrac{\partial}{\partial z}$ and by Proposition~\ref{kernel} the horizonal part $H_{(\tau,z)}\subset T_{(\tau,z)}$ is spanned by $\dfrac{\partial}{\partial\tau}+\beta_2\dfrac{\partial}{\partial z}$. Then, the decomposition of $v$ with respect to $T=V\oplus H$ is
$$
v=v_\tau\dfrac{\partial}{\partial \tau}+v_z \dfrac{\partial}{\partial z}=
\left((v_z-\beta_2v_\tau)\dfrac{\partial}{\partial z}\right)
+v_\tau\left(\dfrac{\partial}{\partial\tau}+\beta_2\dfrac{\partial}{\partial z}\right).
$$
Thus, in terms of the coordinates $(\tau,z,v_\tau,v_z)$ on $T$, the projection $\Pi_V:T\rightarrow V$, as an endomorphism of $T$, is given by
$$
v=(\tau,z,v_\tau,v_z)\mapsto \Pi_V(v)=(\tau,z,0,v_z-\beta_2v_\tau).
$$
Equivalently, as a section on $T^*\otimes V$,
\begin{equation}\label{projection eq}
\Pi_V(v)=(dz-\beta_2d\tau)\otimes\dfrac{\partial}{\partial z}.
\end{equation}
Consequently,
\begin{equation}\label{dbar projection eq}
\bar\partial\Pi_V=-\bar\partial\beta_2\otimes d\tau\otimes\dfrac{\partial}{\partial z}.
\end{equation}

For an elliptic surface $\pi:\mathcal E\rightarrow B$ obtained through a classifying map $f: B\rightarrow X_\Gamma$, there is a holomorphic map $f^\sharp:\mathcal E^0\rightarrow\mathcal M^0_\Gamma$ such that $\pi_\Gamma\circ f^\sharp=f\circ\pi$.
We can use $f^\sharp$ to pull back the Betti coordinate functions and the semi-definite K\"ahler form $\mu$ to $\mathcal E^0$. The pullback $(f^\sharp)^*\mu$ remains semi-definite and everywhere non-zero on $\mathcal E^0$ since the restriction of $f^\sharp$ on each regular fiber is an isomorphism. Consequently, on the dense open subset $\mathcal E^0$, there is a similar decomposition for the holomorphic tangent bundle, as described above.

\subsection{Verticality of a section of an elliptic surface}\label{eta section}

In this section, for the simplicity of notation, we will work on the special case of an elliptic modular surface $\pi_\Gamma:\mathcal M_\Gamma\rightarrow X_\Gamma$. If $\pi:\mathcal E\rightarrow B$ is an elliptic surface with a classifying map $f:B\rightarrow X_\Gamma$, then any section $\sigma$ of $\mathcal E$ will canonically give a holomorphic map $\Sigma:B\rightarrow\mathcal M_\Gamma$ such that $\pi_\Gamma\circ\Sigma=f$. It will be evident that by pulling back the relevant objects, the definitions can be carried over to $\mathcal E$ and the related results obtained here for sections of $\mathcal M_\Gamma$ will hold correspondingly for sections of $\mathcal E$.

\begin{definition}\label{verticality def}
Let $\sigma:X_\Gamma\rightarrow\mathcal M_\Gamma$ be a holomorphic section and $d\sigma:TX_\Gamma\rightarrow \sigma^*T(\mathcal M_\Gamma)$ be its differential. Define the \textit{verticality} of $\sigma$ as
$$
\eta_\sigma:=\Pi_V\circ d\sigma|_{T(X^0_\Gamma)}: T(X^0_\Gamma)\rightarrow\sigma^*V.
$$
\end{definition}
Thus, $\eta_\sigma$ is a real-analytic section of the holomorphic line bundle $T^*(X^0_\Gamma)\otimes\sigma^*V$ on $X^0_\Gamma$. A general version of it for families of abelian varieties has been used by Mok~\cite{Mo} and Mok-To~\cite{MT} to prove the finiteness of the Mordell-Weil groups of the Kuga's families of abelian varieties. Geometrically, $\eta_\sigma$ measures how far $\sigma$ deviates from a torsion section. In fact, we have

\begin{proposition}\label{horizontal prop}
$\eta_\sigma\equiv 0$ if and only if $\sigma$ is a torsion section.
\end{proposition}
\begin{proof}
Since the horizontal part $H\subset T$ consists of the kernels of $\mu$, we deduce from Proposition~\ref{2-form prop} that $\eta_\sigma\equiv 0$ precisely when $\sigma^*(d\beta_1\wedge d\beta_2)\equiv 0$ and thus the Betti coordinates are locally constant on the image of $\sigma$. 
Then the theorem of Manin~\cite{Ma} implies that $\sigma$ is a torsion section.
\end{proof}

For a holomorphic section $\sigma: X_\Gamma\rightarrow \mathcal M_\Gamma$, the local pullback $\beta:=(\sigma^*\beta_1,\sigma^*\beta_2)$ is called the \textit{Betti map} of $\sigma$. Since the construction of $(\beta_1,\beta_2)$ involves a choice of abelian logarithm on $\mathcal M^0_\Gamma$, so does the Betti map $\beta$, but the vanishing order of $\beta$ at any point $b\in B^0$ is independent of such choice and is intrinsic to the section $\sigma$. 

\begin{definition}[\cite{CDMZ}]
The \textit{multiplicity} of a Betti map $\beta$ at $b$ is defined to be the smallest positive integer $m(b)$ such that the partial derivatives of $\sigma^*\beta_1,\sigma^*\beta_2$ at $b$ vanish up to order $m(b)-1$. We will also call $m(b)$ the Betti multiplicity of $\sigma$ at $b$.
\end{definition}

If $\eta_\sigma$ is not identically equal to zero, then its zeros are precisely the points at which the Betti multiplicity of $\sigma$ is at least 2, which can be seen by the following relationship between $\eta_\sigma$ and the Betti map $(\sigma^*\beta_1,\sigma^*\beta_2)$.

\begin{proposition}\label{eta formula}
Let $\tau$ be a local coordinate near a point $p\in X^0_\Gamma$ coming from $\mathcal H$ and $(\tau, z)$ be local coordinates near $\sigma(p)\in\mathcal M^0_\Gamma$ coming from $\mathcal H\times\mathbb C$. Let $(\beta^\sigma_1,\beta^\sigma_2):=(\sigma^*\beta_1,\sigma^*\beta_2)$, where $\beta_1,\beta_2$ are the Betti coordinates such that $z=\beta_1+\beta_2\tau$. Then, 
$$
\eta_\sigma
=2i\im\tau\partial\beta_2^\sigma\otimes\sigma^*\left(\dfrac{\partial}{\partial z}\right)
=\dfrac{2\im\tau}{i\bar\tau}\partial\beta_1^\sigma\otimes\sigma^*\left(\dfrac{\partial}{\partial z}\right).
$$
\end{proposition}
\begin{proof}
The first equation follows directly from Proposition~\ref{beta de} and Eq.(\ref{projection eq}) (Section~\ref{decomposition section}). Since $\sigma$ is holomorphic, we have $\bar\partial(\beta_1^\sigma+\beta_2^\sigma\tau)=0$, which gives $\partial\beta_1^\sigma=-\bar\tau\partial\beta_2^\sigma$ and the second equation follows.
\end{proof}

From how the group $\mathrm{SL}(2,\mathbb Z)\ltimes\mathbb Z^2$ acts on $\mathcal H\times\mathbb C$, it is evident that the relative tangent bundle $V$ on $\mathcal M^0_\Gamma=(\mathcal H\times\mathbb C)/(\Gamma\ltimes\mathbb Z^2)$ is just the quotient of $\mathcal H\times\mathbb C\times\mathbb C$ under the following action of $\Gamma\ltimes\mathbb Z^2$:
$$
(\tau,z,v)\mapsto \left(\dfrac{a\tau+b}{c\tau+d}, \,\,\dfrac{z}{c\tau+d}+\alpha+\beta\tau,\,\,
\dfrac{v}{c\tau+d}\right),
$$
where $\begin{pmatrix}a&b\\c&d\end{pmatrix}\in\Gamma$ and $(\alpha,\beta)\in\mathbb Z^2$.
In particular, we see that the transition functions of the holomorphic tangent bundle $T(X^0_\Gamma)$ of $X^0_\Gamma=\mathcal H/\Gamma$ are precisely the squares of those of the pullback bundle $\sigma^*V$. That is, $T(X^0_\Gamma)\cong(\sigma^*V)^2$ as holomorphic line bundles on $X^0_\Gamma$. Let $p\in X_\Gamma^0$, we make the following explicit identification, locally given by $\sigma^*\left(\dfrac{\partial}{\partial z}\otimes \dfrac{\partial}{\partial z}\right)=\dfrac{\partial}{\partial \tau}$, where $\dfrac{\partial}{\partial z}$ and $\dfrac{\partial}{\partial \tau}$ are vector fields near $\sigma(p)$ and $p$ respectively given by local liftings to the universal covers. It is evident that this identification is independent of the liftings and hence global. To simplify the notations, we write $\left(\dfrac{\partial}{\partial \tau}\right)^{\frac{1}{2}}:=\sigma^*\left(\dfrac{\partial}{\partial z}\right)$ and similarly for the dual $\left(d\tau\right)^{\frac{1}{2}}:=\sigma^*dz$.

Let $K=T^*(X^0_\Gamma)$ be the canonical line bundle of $X^0_\Gamma$. In particular, $K\cong(\sigma^*V)^{-2}$ and we will thus write $(\sigma^*V)^{-1}= K^{\frac{1}{2}}$ and $T^*(X^0_\Gamma)\otimes\sigma^*V=K\otimes K^{-\frac{1}{2}}=K^{\frac{1}{2}}$. Therefore, we can now regard $\eta_\sigma$ as a real-analytic section of $K^{\frac{1}{2}}$ on $X_\Gamma^0$.

Consider the metric $g$ on $X^0_\Gamma$ given by the invariant K\"ahler form on $\mathcal H$
$$
\omega=-i\partial\bar\partial\log\im\tau
=\frac{id\tau\wedge d\bar\tau}{4(\im\tau)^2}.
$$
Denote the conjugate bundle of $K$ by $\overline K$. The reciprocal of $g$ is then a Hermitian metric, denoted by $g_\star$, on the canonical line bundle $K$. We will regard $g_\star$ as a section of $K^{-1}\otimes\overline K^{-1}$. Similarly, we have the sections ${g_\star}^{\frac{m}{2}}$ of $K^{-\frac{m}{2}}\otimes\overline K^{-\frac{m}{2}}$, for each $m\in\mathbb Z$.


As mentioned above, $\eta_\sigma$ can be naturally regarded as a real analytic section of $K^{\frac{1}{2}}$. It follows that $\bar\partial\eta_\sigma$ is a real analytic section of $\overline K\otimes K^{\frac{1}{2}}$. We have

\begin{proposition}\label{dbar eta}
$\bar\partial\eta_\sigma=-i\,\overline{\eta_\sigma}\otimes g_\star^{-\frac{1}{2}}$.
\end{proposition}
\begin{proof}
Using the notations above and in terms of the standard local coordinate $\tau$, by Proposition~\ref{eta formula}, we have 
$\eta_\sigma=2i\im\tau\partial\beta^\sigma_2\otimes\left(\dfrac{\partial}{\partial\tau}\right)^{\frac{1}{2}}$. On the other hand, using $g_\star^{-\frac{1}{2}}=\dfrac{\left(d\tau\right)^{\frac{1}{2}}\otimes\left(d\bar\tau\right)^{\frac{1}{2}}}{2\im\tau}$ and pulling back Eq.(\ref{dbar projection eq}) of Section~\ref{decomposition section} by $\sigma$, 
\begin{eqnarray*}
\bar\partial\eta_\sigma&=&\bar\partial(\Pi_V\circ d\sigma)=(\bar\partial\Pi_V)\circ d\sigma=-\bar\partial\beta^\sigma_2\otimes d\tau\otimes\left(\dfrac{\partial}{\partial\tau}\right)^{\frac{1}{2}}\\
&=&-\bar\partial\beta^\sigma_2\otimes (d\tau)^{\frac{1}{2}}=-i\,\overline{\eta_\sigma}\otimes g_\star^{-\frac{1}{2}}.
\end{eqnarray*}
\end{proof}

The first order PDE above satisfied by $\eta_\sigma$ has already been established by Mok-To~\cite{MT}. Although the PDE stated in~\cite{MT} was formulated and stated in a slightly different form, one can readily check that it is equivalent to the one here. This is a first order real-linear PDE of $\eta_\sigma$ and was used to show that $\eta_\sigma$ satisfies an eigenequation of the Laplace operator, which is the crux of Mok-To's proof of the finiteness of the Mordell-Weil group for $\mathcal M_\Gamma$. It turns out that this first order PDE is also important for the study of Betti multiplicity and we will make use of it to prove that the zero set of a non-trivial $\eta_\sigma$ is discrete and derive an integral formula which counts the total Betti multiplicity of $\sigma$.

We know by Proposition~\ref{horizontal prop} that $\eta_\sigma\equiv 0$ implies that $\sigma$ is a torsion section. Given a non-torsion $\sigma$, then we are primarily interested in the zeros of $\eta_\sigma$. From how $\eta_\sigma$ is constructed from the section $\sigma:X_\Gamma\rightarrow\mathcal M_\Gamma$, we see that $\eta_\sigma$ is zero precisely at the points at which $\sigma$ is tangent to a local horizontal section, i.e. a local section given $\tau\mapsto (\tau, a+b\tau)$, where $a,b\in\mathbb R$. 

\begin{proposition}\label{isolation}
The zero set of a non-trivial $\eta_\sigma$ is a discrete set in $X^0_\Gamma$.
\end{proposition}
\begin{proof}
Suppose $w=0$ is a zero of $\eta_\sigma$ with respect to a local coordinate $w$ on $X^0_\Gamma$.  By choosing local holomorphic bases for the bundles $K^{\frac{1}{2}}$ and $\overline K\otimes K^{\frac{1}{2}}$, we regard $\eta_\sigma$ as a complex-valued real-analytic function near $w=0$ and identify $\bar\partial\eta_\sigma$ with $\partial\eta_\sigma/\partial\overline w$. 
By splitting the power series expansion of $\eta_\sigma$, we write
$$
\eta_\sigma(w)=f(w)+\overline{h(w)}+\psi(w,\overline w),
$$ 
where $f$, $h$ are holomorphic and $\psi$ is real-analytic in a neighborhood of $w=0$, such that $\psi$ does not contain any pure terms $w^k$ nor $\overline w^k$, $k\in\mathbb N^+$. That is, $f(w)$ (resp. $\overline{h(w)}$) only contains the powers of $w$ (resp. $\overline w$), and $\psi(w,\overline w)$ contains mixed terms only. Now by Proposition~\ref{dbar eta}, there is a non-vanishing local real-analytic function $G(w,\overline w)$ such that $\partial\eta_\sigma/\partial\overline w=G\,\overline{\eta_\sigma}$, thus
$$
\dfrac{\partial\overline h}{\partial\overline w}+\dfrac{\partial\psi}{\partial\overline w}=G(\overline f+h+\overline\psi).
$$
By comparing the terms of the lowest power in $w$ that do not contain powers of $\overline w$, we deduce that the vanishing order of $\psi$ is one greater than that of $h$. (In the case where $h\equiv 0$, we then have $\psi=f\equiv 0$.) Similarly, by comparing the terms of the lowest power in $\overline w$ that do not contain powers of $w$, we get that the vanishing order of $h$ is one greater than that of $f$. Therefore, if $\ell\geq 1$ is the vanishing order of $f$, then we can write $\eta_\sigma=w^\ell\eta^\sharp$ for some non-vanishing local continuous function $\eta^\sharp$. Hence, if $\eta_\sigma$ is not identically zero, then $w=0$ is an isolated zero for $\eta_\sigma$.
\end{proof}

Let $\nabla$ be the $(1,0)$-part of the Hermitian connection associated to ${g_\star}^{\frac{1}{2}}$ on $K^{\frac{1}{2}}$. Then $\nabla\eta_\sigma$ is a section of $K^{\frac{3}{2}}$.

\begin{proposition}\label{d eta}
$\nabla\eta_\sigma$ is holomorphic.
\end{proposition}
\begin{proof}
This can be obtained by directly calculating the covariant derivative $\nabla\eta_\sigma$, which has been done in~\cite{MT}.
\end{proof}

We mentioned that the definitions and results in this section can be carried over to any elliptic surface with a classifying map. As an example, we generalize the definition of verticality, as follows.  Suppose $\pi:\mathcal E\rightarrow B$ is an elliptic surface over a complex projective curve $B$ given by a classifying map $f: B\rightarrow X_\Gamma$. Then, given a holomorphic section $\sigma:B\rightarrow\mathcal E$, there exists a unique a holomorphic map $\Sigma:B\rightarrow\mathcal M_\Gamma$ such that $\pi_\Gamma\circ\Sigma=f$, where $\pi_\Gamma:\mathcal M_\Gamma\rightarrow X_\Gamma$ is the canonical projection. We define 
$$
\eta_\sigma:=(\Sigma^*\Pi_V)\circ (d\Sigma|_{T(B^0)}): T(B^0)\rightarrow\Sigma^*V|_{\mathcal M^0_\Gamma}=f^*(K|_{X^0_\Gamma})^{-\frac{1}{2}},
$$
where $B^0=B\setminus S$ and $S=f^{-1}(X_\Gamma\setminus X^0_\Gamma)$.

\subsection{Verticality and Betti multiplicity}\label{modular section}

Now we are going to discuss how one can study the Betti multiplicity of $\sigma$ using $\eta_\sigma$. For a heuristic purpose, we will again work on an elliptic modular surface $\mathcal M_\Gamma$ even though it is already known that there are no non-torsion sections for $\mathcal M_\Gamma$. In the next section, we will count the Betti multiplicity for the general case with a somewhat different formulation. However, we will be able to see that the analysis there is hinted by the study of verticality in the case of elliptic modular surfaces.  

Suppose $\sigma$ is a hypothetical non-torsion holomorphic section for $\mathcal M_\Gamma$. In what follows, we will write $\eta$ instead of $\eta_\sigma$ to simplify the notation.  Denote by $\|\eta\|^2$  the norm squared with respect to $\mu$ (which can be regarded as a Hermitian metric on $V$ since the null space of $\mu$ is transversal to $V$). Let $w=0$ be a zero of $\eta$ in a local coordinate $w$ on $X^0_\Gamma$. By Proposition~\ref{isolation}, the zero(s) of $\eta$ are isolated on $X^0_\Gamma$. In addition, the vanishing order of $\eta$ should be the Betti multiplicity minus 1 by Proposition~\ref{eta formula} and from the proof of Proposition~\ref{isolation}, we know that $\|\eta\|^2=|w|^{2m-2}\varphi$ locally, where $m\geq 2$ is the Betti-multiplicity and $\varphi$ is a non-vanishing local continuous function. 

Recall that on $X^0_\Gamma$, we have $\sigma^*V=K^{-\frac{1}{2}}$ and $\eta$ can be regarded as a real-analytic section of $K^{\frac{1}{2}}$. In what follows, we will use the same symbols $\nabla$ and $\overline\nabla$ to respectively denote the $(1,0)$-part and $(0,1)$-part of the Hermitian connection associated to ${g_\star}^{\frac{m}{2}}$ on $K^{\frac{m}{2}}$ for each $m\in\mathbb Z$ and also for the conjugate connection on $\overline K^{\frac{m}{2}}$. Here the Hermitian metric $g_\star$ on $K$ is regarded as a section of $K^{-1}\otimes\overline K^{-1}$. Then, we can write $\|\eta\|^2=\eta\otimes\overline{\eta}\otimes{g_\star}^{\frac{1}{2}}$. 

At a point where $\|\eta\|^2\neq 0$, we have
\begin{eqnarray*}
\bar\partial\log\|\eta\|^2
&=&\dfrac{\bar\partial\|\eta\|^2}{\|\eta\|^2}\\
&=&\dfrac{
\overline\nabla\eta\otimes\overline{\eta}\otimes{g_\star}^{\frac{1}{2}}
+\eta\otimes\overline{\nabla\eta}\otimes{g_\star}^{\frac{1}{2}}
}
{\|\eta\|^2}\\
&=&\dfrac{\overline\nabla\eta}{\eta}+
\dfrac{\overline{\nabla\eta}}{\overline\eta}
\end{eqnarray*}
Moreover, using Proposition~\ref{dbar eta},
\begin{eqnarray*}
\nabla\left(\dfrac{\overline\nabla\eta}{\eta}\right)
&=&\nabla\left(\dfrac{-i\,\overline\eta\otimes g_\star^{-\frac{1}{2}}}{\eta}\right)\\
&=&\dfrac{-i\,\nabla\overline\eta\otimes g_\star^{-\frac{1}{2}}}{\eta}+
\dfrac{i\,\nabla\eta\otimes\overline\eta\otimes g_\star^{-\frac{1}{2}}}{\eta^2}\\
&=&\dfrac{-i\,\overline{\overline\nabla\eta}\otimes g_\star^{-\frac{1}{2}}}{\eta}-
\dfrac{\nabla\eta}{\eta}\otimes\dfrac{\overline\nabla\eta}{\eta}\\
&=&g_\star^{-1}-\dfrac{\nabla\eta}{\eta}\otimes\dfrac{\overline\nabla\eta}{\eta}
\end{eqnarray*}
and by Proposition~\ref{d eta},
$$
\nabla\left(\dfrac{\overline{\nabla\eta}}{\overline\eta}\right)=
\overline{\overline\nabla\left(\dfrac{\nabla\eta}{\eta}\right)}
=-\overline{\nabla\eta\otimes\dfrac{\overline\nabla\eta}{\eta^2}}
=-\overline{\dfrac{\nabla\eta}{\eta}\otimes\dfrac{\overline\nabla\eta}{\eta}}.
$$

Therefore, if we regard $\bar\partial\log\|\eta\|^2$ as a section on $\overline K$, then at a point where $\|\eta\|^2\neq 0$, we have
$$
\nabla(\bar\partial\log\|\eta\|^2)
=g_\star^{-1}-\dfrac{\nabla\eta}{\eta}\otimes\dfrac{\overline\nabla\eta}{\eta}
-\overline{\dfrac{\nabla\eta}{\eta}\otimes\dfrac{\overline\nabla\eta}{\eta}}
$$
as sections of $K\otimes\overline K$. For a section of $\alpha$ of $K\otimes\overline K$, and $\alpha(\tau)=\alpha_0(\tau) d\tau\otimes d\bar\tau$ in terms of local coordinates, we use  $\alpha^\wedge$ to denote the $(1,1)$-form $\alpha_0(\tau)d\tau\wedge d\bar\tau$, which is well defined (independent of coordinates). Then, we have $i(\nabla(\bar\partial\log\|\eta\|^2))^\wedge=i\partial\bar\partial\log\|\eta\|^2$ and 
$(ig_\star^{-1})^\wedge=\omega$, where $\omega$ is the K\"ahler form descending from the invariant form $\dfrac{id\tau\wedge d\bar\tau}{4(\im\tau)^2}$ on $\mathcal H$. If we also let $\chi:=\left(i\dfrac{\nabla\eta}{\eta}\otimes\dfrac{\overline\nabla\eta}{\eta}\right)^\wedge$, then
\begin{equation}\label{iddbarlogeta}
\dfrac{i}{2\pi}\partial\bar\partial\log\|\eta\|^2=\dfrac{1}{2\pi}\left(\omega-\chi-\overline\chi\right),
\end{equation}
which is a real-analytic $(1,1)$-form on $X^0_\Gamma\setminus\mathfrak B$, where $\mathfrak B\subset X^0_\Gamma$ is the zero set of $\eta$, or equivalently, the set of points at which the Betti multiplicity is at least 2.

Let $S=X_\Gamma\setminus X^0_\Gamma$. For every point $x\in S\cup\mathfrak B$, choose a local coordinate chart $w\in\Delta:=\{w\in\mathbb C:|w|<1\}$ such that $w(x)=0$.
Regard $\overline\nabla\eta/\eta$ as a $(0,1)$-form on $X_\Gamma\setminus(S\cup\mathfrak B)$, then by Proposition~\ref{dbar eta} it is of constant norm with respect to $\omega$. In particular, if $x\in\mathfrak B$, the integral $\int_{\partial\Delta_\epsilon} \overline\nabla\eta/\eta$ will tend to zero when $\epsilon\rightarrow 0$, where $\Delta_\epsilon=\{w\in\Delta: |w|<\epsilon\}$. On the other hand, using the standard compactifying coordinate $q=e^{2\pi i\tau/k}$ (c.f. Section~\ref{toroidal}) near the cusp $x_\infty$ corresponding to $i\infty$, where $k$ is that in $\Gamma(k)=\Gamma$, and using Proposition~\ref{dbar eta} again, we get
$$
\left|\int_{\partial\Delta_\epsilon}\dfrac{\overline\nabla\eta}{\eta}\right|\leq
\int_{\partial\Delta_\epsilon}\dfrac{|dq|}{2|q\log|q||}.
$$
It then follows readily that $\int_{\partial\Delta_\epsilon} \overline\nabla\eta/\eta$ also goes to zero as $\epsilon\rightarrow 0$. For other cusps, it suffices to exploit the $\mathrm{SL}(2,\mathbb Z)$ action to conclude that the same holds true for every point $x\in S$. Consequently, if we let 
$\displaystyle X_{\Gamma,\epsilon}:=X_\Gamma\setminus\bigcup_{x\in S\cup\mathfrak B}\overline{\Delta_\epsilon(x)}$, where $\Delta_\epsilon(x)$ is the $\Delta_\epsilon$ defined above for each $x$, then by Stokes' theorem, we get
$$
0=\lim_{\epsilon\rightarrow 0}\int_{X_{\Gamma,\epsilon}}id\left(\dfrac{\overline\nabla\eta}{\eta}\right)=\lim_{\epsilon\rightarrow 0}\int_{X_{\Gamma,\epsilon}}i\partial\left(\dfrac{\overline\nabla\eta}{\eta}\right)=\lim_{\epsilon\rightarrow 0}\int_{X_{\Gamma,\epsilon}}(\omega-\chi).
$$
In particular, we get 
\begin{equation}\label{integrable prop}
\int_{X^0_\Gamma}\omega=\lim_{\epsilon\rightarrow 0}\int_{X_{\Gamma,\epsilon}}\omega=\lim_{\epsilon\rightarrow 0}\int_{X_{\Gamma,\epsilon}}\chi=\lim_{\epsilon\rightarrow 0}\int_{X_{\Gamma,\epsilon}}\overline\chi
\end{equation}
as the integral of $\omega$ is a real number. 
In addition, as shown in the proof of Proposition~\ref{isolation}, if $w=0$ is a zero of $\eta$ in terms of a local coordinate $w$, we have $\|\eta\|^2=|w|^{2m-2}\varphi$ for some non-vanishing local continuous function, where $m$ is the Betti multiplicity of $\sigma$ at $w=0$. Then, for $w\neq 0$,
$$
\dfrac{i}{2\pi}\bar\partial\log\|\eta\|^2=\dfrac{i(m-1)}{2\pi}\dfrac{d\overline w}{\overline w}
+\dfrac{i}{2\pi}\dfrac{\bar\partial\varphi}{\varphi}
$$
and thus
\begin{equation}\label{loop eq}
\int_{\partial\Delta_\epsilon}\dfrac{i}{2\pi}\bar\partial\log\|\eta\|^2=m-1+
\int_{\partial\Delta_\epsilon}\dfrac{i}{2\pi}\dfrac{\bar\partial\varphi}{\varphi}
\end{equation}
for sufficiently small $\epsilon$, where $\partial\Delta_\epsilon$ is given with the anti-clockwise orientation. In next section, which deals with a general elliptic surface given by a classifying map, we will show that the last integral will go to zero when $\epsilon\rightarrow 0$. Now together with Eq.(\ref{iddbarlogeta}) and Eq.(\ref{integrable prop}), by applying Stokes' theorem on $\dfrac{i}{2\pi}\partial\bar\partial\log\|\eta\|^2=\dfrac{i}{2\pi}d\bar\partial\log\|\eta\|^2$, we have 
$$
\frac{1}{2\pi}\int_{X^0_\Gamma}\omega=\sum_{b\in\mathfrak B}(m_b-1)+\lim_{\epsilon\rightarrow 0}\sum_{b\in S}\int_{\partial\Delta_\epsilon(b)} \dfrac{i}{2\pi}\bar\partial\log\|\eta\|^2.
$$

Also in the next section, we will study the last term in the equation above, which is related to the asymptotic behavior of $\|\eta\|^2$ near the points of bad reduction. We will do this by using a potential function defined on the projectivized tangent bundle of $\mathcal M^0_\Gamma$, with which it is easier to deal with the more general situation of elliptic surfaces given by possibly ramified classifying maps into the modular curves. In any case, the analysis above suggests that the integral of the invariant metric over the points of good reduction on the base curve is related to the total Betti multiplicity.


\section{Counting the Betti multiplicity of a section}\label{counting big section}

We will now construct an invariant function $\Psi$ on the projectivized tangent bundle $\mathbb PT(\mathcal H\times\mathbb C)$, which thus descends to the $\mathbb PT(\mathcal M^0_\Gamma)$. It will then be shown that $\hat\sigma^*\Psi$ is just $\|\eta\|^2$, where $\hat\sigma: X^0_\Gamma\rightarrow\mathbb PT(\mathcal M^0_\Gamma)$ is the tautological lifting of a section $\sigma:X_\Gamma\rightarrow\mathcal M_\Gamma$. Using $\Psi$, we can efficiently study the Betti multiplicities of the sections of a general elliptic surface given by a classifying into $X_\Gamma$.

\subsection{The invariant function $\Psi$ on $\mathbb PT(\mathcal H\times\mathbb C)$}

To simplify the notation, we will let $T:=T(\mathcal H\times\mathbb C)$ be the holomorphic tangent bundle of $\mathcal H\times\mathbb C$. Let $(\tau,z)\in \mathcal H\times\mathbb C$ and $v\in T_{(\tau,z)}$, where $T_{(\tau,z)}$ denotes the holomorphic tangent space at $(\tau, z)$. We can write $v=v_\tau\dfrac{\partial}{\partial\tau}+v_z\dfrac{\partial}{\partial z}$, for some $v_\tau,v_z\in\mathbb C$. Write also $z=\beta_1+\beta_2\tau$, where $\beta_1,\beta_2\in\mathbb R$ are the Betti coordinates. In particular, $\beta_2=\im z/\im\tau$.
In terms of the above coordinates on $T$, the matrix representing the semi-definite \Kahler form $\mu$ with respect to the basis $\left(\dfrac{\partial}{\partial\tau},\dfrac{\partial}{\partial z}\right)$ is
$$
g_\mu=\dfrac{1}{2\im\tau}
\begin{pmatrix}
\beta_2^2&-\beta_2\\
-\beta_2&1
\end{pmatrix}
$$
Thus, the norm squared of $v\in T$ measured against $g_\mu$ is 
$$
\|v\|^2_\mu=\dfrac{|v_z-\beta_2v_\tau|^2}{2\im\tau},
$$
which will be regarded as a function on $T$. Since $\mu$ is invariant under the action of $\mathrm{SL}(2,\mathbb R)\ltimes\mathbb R^2$, it follows that $\|v\|_\mu^2$ is an invariant function on $T$. Similarly, the norm squared $$\|v\|_\nu^2=\dfrac{|v_\tau|^2}{2(\im\tau)^2},$$ where $\nu$ is the pullback of the Poincar\'e metric on $\mathcal H$ to $\mathcal H\times\mathbb C$, is also invariant on $T$. By take the quotient of these two functions, we can now define an invariant function $\Psi:\mathbb PT\rightarrow [0,+\infty]$ by
$$
\Psi(\tau,z,[v_z,v_\tau]):=\dfrac{\|v\|_\mu^2}{\|v\|_\nu^2}=\dfrac{|v_z-\beta_2v_\tau|^2\im\tau}{|v_\tau|^2}=
\dfrac{\left|v_z\im\tau-v_\tau\im z\right|^2}{|v_\tau|^2\im\tau}.
$$
Note that $\Psi$ is well defined as a value in $[0,+\infty]$ since $\im\tau>0$ and $v_\tau$, $v_z$ cannot be both zero on $\mathbb PT$. 

\begin{proposition}\label{equivalence}
Let $\sigma$ be a holomorphic section of $\mathcal M_\Gamma$ and $\hat\sigma: X_\Gamma\rightarrow\mathbb PT(\mathcal M_\Gamma)$ be the tautological lift of $\sigma$.
Then, $\hat\sigma^*\Psi=\|\eta_\sigma\|^2$ on $X^0_\Gamma$.
\end{proposition}
\begin{proof}
Let $p\in X^0_\Gamma$ and $0\neq v\in T_p(X^0_\Gamma)$, then 
$$
\|d\sigma(v)\|_\mu^2=\|\eta_\sigma(v)\|_\mu^2=\|\eta_\sigma\|^2(p)\|v\|_g^2=\|\eta_\sigma\|^2(p)\|d\sigma(v)\|_\nu^2,
$$ 
where $g$ is the Poincar\'e metric on $X^0_\Gamma$ descending from $\mathcal H$. Thus,
$$
\|\eta_\sigma\|^2(p)=\dfrac{\|d\sigma(v)\|_\mu^2}{\|d\sigma(v)\|_\nu^2}=\left(\hat\sigma^*\Psi\right)(p).
$$
\end{proof}

\begin{proposition}\label{isolated prop}
Let $h:\Delta\rightarrow\mathcal H\times\mathbb C$ be a local holomorphic curve such that $h^*\beta_2$ and $h^*\tau$ are non-constant, where $\Delta=\{w\in\mathbb C:|w|<1\}$. Denote by $\hat h:\Delta\rightarrow\mathbb PT$ the tautological lifting of $h$. Let $r$ and $m$ be the vanishing order of 
$h^*\tau-h^*\tau(0)$ and $h^*\beta_2-h^*\beta_2(0)$ at $w=0$ respectively, then 
$$
\hat h^*\Psi(w)=|w|^{2(m-r)}\psi(w)
$$
in a neighborhood of $w=0$, where $\psi$ is a non-vanishing local continuous function. In particular, the zeros and poles of $\hat h^*\Psi$ are isolated.
\end{proposition}

\begin{proof}
Write the local holomorphic curve as $h(w)=(\tau(w),z(w))$, $w\in\Delta$. For simplicity, we write $\beta_2(w):=h^*\beta_2(w)$. Since
$$
\begin{matrix}[ccl]
\beta_2(w)-\beta_2(0)&=&\dfrac{z(w)-\overline{z(w)}}{\tau(w)-\overline{\tau(w)}}-\beta_2(0)\\
\\
&=&\dfrac{z(w)-\beta_2(0)\tau(w)-\overline{z(w)-\beta_2(0)\tau(w)}}
{\tau(w)-\overline{\tau(w)}}
\end{matrix}
$$
vanishes to the order $m$ and $\tau(w)-\overline{\tau(w)}$ is always non-zero, it follows that $$
z(w)-\beta_2(0)\tau(w)=b+w^m\rho(w)
$$ 
for some $b\in\mathbb R$ and some holomorphic function $\rho$ such that $\rho(0)\neq 0$. (In fact, $b$ is just $\beta_1(0):=\beta_1(h(0))$.) Then,
$$
z'(w)=\beta_2(0)\tau'(w)+w^{m-1}\rho^\sharp(w)
$$
for some holomorphic $\rho^\sharp$ such that $\rho^\sharp(0)\neq 0$.
Now, in a punctured neighborhood of $w=0$,
\begin{eqnarray*}
\Psi(\hat h(w))
&=&\dfrac{\left|z'(w)-\tau'(w)\beta_2(w)\right|^2\im\tau(w)}{|\tau'(w)|^2}\\
&=&\dfrac{\left|\tau'(w)(\beta_2(0)-\beta_2(w))+w^{m-1}\rho^\sharp(w)\right|^2\im\tau(w)}
{|\tau'(w)|^2}\\
&=&\dfrac{\left|\tau'(w)\im(w^m\rho(w))/\im\tau(w)+w^{m-1}\rho^\sharp(w)\right|^2\im\tau(w)}
{|\tau'(w)|^2}\\
&=&\dfrac{\left|\tau'(w)\im(w^m\rho(w))+w^{m-1}\rho^\sharp(w)\im\tau(w)\right|^2}
{|\tau'(w)|^2\im\tau(w)}\\
&=&\dfrac{|w|^{2m-2}\left|\dfrac{\tau'(w)}{2i}\left(w\rho(w)-\dfrac{\overline w^m\overline\rho(w)}{w^{m-1}}\right)+\rho^\sharp(w)\im\tau(w)\right|^2}
{|\tau'(w)|^2\im\tau(w)}.
\end{eqnarray*}
Thus,
$$
\Psi(\hat h(w))=\dfrac{|w|^{2m-2}\chi(w)}{|\tau'(w)|^2}
=\dfrac{|w|^{2m-2}}{|w|^{2r-2}}\psi(w)
=|w|^{2m-2r}\psi(w),
$$
where $\chi$, $\psi$ are some non-vanishing local continuous functions near $w=0$.
\end{proof}

\begin{corollary}\label{isolated cor 1}
Let $\sigma:B\rightarrow\mathcal E$ be a non-torsion section of an elliptic surface given by some classifying map $f$ from $B$ to some modular curve $X_{\Gamma}$. Let $\hat\Sigma:B\rightarrow\mathbb PT(\mathcal M_{\Gamma})$ be the tautological lifting of the induced map $\Sigma:B\rightarrow\mathcal M_{\Gamma}$. Let $S\subset B$ be the points of bad reduction. Then, for every $b\in B\setminus S$, there exist a coordinate chart $w\in\Delta$ near $b$ with $w(b)=0$ and a non-vanishing continuous function $\psi$ on $\Delta$, such that 
$$
\hat\Sigma^*\Psi(w)=|w|^{2(m_b-r_b)}\psi(w),
$$
where $m_b$ is the Betti multiplicity of $\sigma$ at $b$ and $r_b$ is the ramification index of $f$ at $b$. In particular, the points on $B\setminus S$ at which the Betti multiplicity of $\sigma$ is at least 2 are isolated.
\end{corollary}


\subsection{Betti multiplicity on the toroidal compactification}\label{toroidal}

Being invariant, $\Psi$ thus descends to a function, which we will still denoted by $\Psi$, on any $\mathbb PT(\mathcal M^0_\Gamma)$, where $\mathcal M^0_\Gamma=(\mathcal H\times\mathbb C)/(\Gamma\ltimes\mathbb Z^2)$. We will now extend it to the compactification $\mathbb PT(\mathcal M_\Gamma)$. Let $\Gamma(k)$ be a principal congruence subgroup, where $k\geq 3$. We first recall some basics of the toroidal compactification $\pi_k:\mathcal M_{\Gamma(k)}\rightarrow X_{\Gamma(k)}$, which can be found in~\cite{AMRT}. At any point on the singular fiber over the cusp $c\in X_{\Gamma(k)}$ corresponding to $i\infty$, there exists a coordinate charts $(\xi,\zeta)\in\mathbb C^2$ such that the singular fiber is defined by $\{\xi\zeta=0\}\subset\mathbb C^2$. In addition, for a point in the regular part $\mathbb C^2\setminus\{\xi\zeta=0\}$, the coordinate transformation from $(\xi,\zeta)$ to the local coordinates $(\tau,z)$ descends from $\mathcal H\times\mathbb C$ is as follows: there exists an integer $n$, unique modulo $k$, such that
$$
\xi\zeta=e^{2\pi i\tau/k}
\eqand
\xi^n\zeta^{n+1}=e^{2\pi iz}.
$$
Furthermore, there is a local coordinate $q$ near the cusp $c$ such that $q(c)=0$ and if $\pi_k:\mathcal M_{\Gamma(k)}\rightarrow X_{\Gamma(k)}$ is the projection map, then $q(\pi_k(\xi,\zeta))=\xi\zeta$. In particular, $q=e^{2\pi i\tau/k}$ in a punctured neighborhood of $c$.

On the tangent bundle $T(\mathcal M^0_{\Gamma(k)})$ near the singular fiber $\pi_k^{-1}(c)$, by computing the coordinate transformation between $(\xi,\zeta,v_\xi,v_\zeta)$ and $(\tau,z,v_\tau,v_z)$ we get
$$
\begin{pmatrix}\im\tau\\ \im z\end{pmatrix}=\dfrac{-1}{2\pi}
\begin{pmatrix}k\log|\xi\zeta|\\ \log|\xi^n\zeta^{n+1}|\end{pmatrix}
\eqand
\begin{pmatrix}v_\tau\\ v_z\end{pmatrix}=\dfrac{1}{2\pi i\xi\zeta}
\begin{pmatrix}k(\zeta v_\xi+\xi v_\zeta)\\ n(\zeta v_\xi+\xi v_\zeta)+\xi v_\zeta\end{pmatrix}.
$$

Consequently,
$$
\begin{matrix}[ccl]
\left|v_z\im\tau-v_\tau\im z\right|^2&=&\dfrac{1}{16\pi^4|\xi\zeta|^2}
\left|
(n(\zeta v_\xi+\xi v_\zeta)+\xi v_\zeta)k\log|\xi\zeta|
-k(\zeta v_\xi+\xi v_\zeta)\log|\xi^n\zeta^{n+1}|
\right|^2
\\
\\
&=&\dfrac{1}{16\pi^4|\xi\zeta|^2}
\left|
k\xi v_\zeta\log|\xi\zeta|-k(\zeta v_\xi+\xi v_\zeta)\log|\zeta|
\right|^2
\\
\\
&=&\dfrac{k^2}{16\pi^4|\xi\zeta|^2}
\left|
v_\zeta\xi \log|\xi|-v_\xi\zeta\log|\zeta|
\right|^2
\end{matrix}
$$
and
$$
|v_\tau|^2=\dfrac{k^2|\zeta v_\xi+\xi v_\zeta|^2}{4\pi^2|\xi\zeta|^2}.
$$

In a neighborhood of a point on the singular fiber over a cusp of $X_{\Gamma(k)}$, the function $\Psi$ on $\mathbb PT(\mathcal M_{\Gamma(k)})$ can be then expressed as
$$
\Psi(\xi,\zeta,[v_\xi,v_\zeta])
=\dfrac{\left|v_z\im\tau-v_\tau\im z\right|^2}{|v_\tau|^2\im\tau}
=\dfrac{\left|v_\zeta\xi\log|\xi|-v_\xi\zeta\log|\zeta|\right|^2}
{-2\pi k|\zeta v_\xi+\xi v_\zeta|^2\log|\xi\zeta|}
$$

\begin{proposition}\label{isolated prop 2}
Let $h:\Delta\rightarrow\mathcal M_{\Gamma(k)}$ be a local holomorphic curve, where $\Delta:=\{w\in\mathbb C:|w|<1\}$, such that $h(0)$ is a regular point of a singular fiber and for every $w\neq 0$, $h(w)$ lies in $\mathcal M^0_{\Gamma(k)}$. Let $\hat h:\Delta\rightarrow\mathbb PT(\mathcal M_{\Gamma(k)})$ be the tautological lifting of $h$, then either $\Psi(\hat h(w))\equiv 0$ or there exist some integer $m\geq 1$ and a non-vanishing local continuous function $\phi(w)$ such that in a neighborhood of $w=0$,
$$
\hat h^*\Psi(w)=|w|^{2(m-1)}(\log|w|)^{\pm 1}\phi(w).
$$ 
\end{proposition}
\begin{proof}
Since $h(0)$ lies on a singular fiber, we can assume that the image of $h$ lies in a coordinate chart $(\xi,\zeta)\in\mathbb C^2$ for the toroidal compactification described above and write $h(w)=(\xi(w),\zeta(w))$, $w\in\Delta$. In the coordinate chart $(\xi,\zeta)$, the singular fiber is defined by the equation $\xi\zeta=0$. Since $h(0)$ is a regular point of the singular point and $\Psi$ is symmetric in $\xi$, $\zeta$, if we let $h(0)=(\xi_0,\zeta_0)$, we may assume that $\xi_0\neq 0$ and $\zeta_0=0$. In addition, $\zeta(w)\not\equiv 0$ as the image of $h$ only intersects the singular fiber at one point. 

We first handle the case where $\xi(w)\not\equiv\xi_0$. Then, there exist positive integers $j,\ell$ such that 
$$
\xi(w)=\xi_0+w^j\tilde\xi(w)
\eqand
\zeta(w)=w^\ell\tilde\zeta(w).
$$
for some holomorphic functions $\tilde\xi$, $\tilde\zeta$ such that $\tilde\xi(0)\neq 0$, 
$\tilde\zeta(0)\neq 0$.  In particular, we also have
$$
\xi'(w)=w^{j-1}\xi^\sharp(w)
\eqand
\zeta'(w)=w^{\ell-1}\zeta^\sharp(w).
$$
for some holomorphic functions $\xi^\sharp$, $\zeta^\sharp$ such that $\xi^\sharp(0)\neq 0$, 
$\zeta^\sharp(0)\neq 0$.
The tautological lifting of $h$ can be written as $\hat h(w)=(\xi(w),\zeta(w),[\xi'(w),\zeta'(w)])$. 
Hence, for $w\neq 0$,
\begin{equation}\label{nf(w)}
\Psi(\hat h(w))=\dfrac{\left|\zeta'\xi\log|\xi|-\xi'\zeta\log|\zeta|\right|^2}
{-2\pi k|\zeta \xi'+\xi \zeta'|^2\log|\xi\zeta|}.
\end{equation}

For the numerator,
$$
\begin{matrix}[ccl]
\left|\zeta'\xi\log|\xi|-\xi'\zeta\log|\zeta|\right|^2
&=&\left|w^{\ell-1}\zeta^\sharp\xi\log|\xi|-w^{j+\ell-1}\xi^\sharp\tilde\zeta\log|w^\ell\tilde\zeta|\right|^2\\
\\
&=&\left|w\right|^{2(\ell-1)}\left|\zeta^\sharp\xi\log|\xi|-w^j\xi^\sharp\tilde\zeta\log|w^\ell\tilde\zeta|\right|^2.
\end{matrix}
$$
If $|\xi_0|\neq 1$, we see that for $w$ close to $0$, 
$$
\left|\zeta'\xi\log|\xi|-\xi'\zeta\log|\zeta|\right|^2=
|w|^{2(\ell-1)}\psi_1(w),
$$
where $\psi_1$ is a non-vanishing local continuous function.
Suppose now $|\xi_0|=1$ and write
$$
\left|\zeta'\xi\log|\xi|-\xi'\zeta\log|\zeta|\right|^2=
\left|w\right|^{2(j+\ell-1)}\left|\zeta^\sharp\xi\dfrac{\log|\xi|}{w^j}-\xi^\sharp\tilde\zeta\log|w^\ell\tilde\zeta|\right|^2.
$$
When $w\rightarrow 0$, we have 
$$
\left|\dfrac{\log|\xi|}{w^j}\right|=\left|\dfrac{\log|\xi_0+w^j\tilde\xi|}{w^j}\right|
=\left|\dfrac{\log|1+\overline{\xi_0}w^j\tilde\xi|}{w^j}\right|
\leq\left|\dfrac{\log(1+\overline{\xi_0}w^j\tilde\xi)}{w^j}\right|,
$$
which is bounded and it follows that for $w$ close to $0$,
$$
\left|\zeta'\xi\log|\xi|-\xi'\zeta\log|\zeta|\right|^2=
|w|^{2(j+\ell-1)}(\log|w|)^2\psi_2(w)
$$
for some non-vanishing local continuous function $\psi_2$. Now let us handle the case where $\xi(w)\equiv\xi_0$. The numerator of $\Psi(\hat h(w))$ in Eq.(\ref{nf(w)}) in this case will be either constantly zero (when $|\xi_0|=1$) or equal to $a|w|^{2(\ell-1)}|\zeta^\sharp|^2$, for some $a>0$ (when $|\xi_0|\neq 1$).

On the other hand, it is easy to see that when $w$ is close to $0$, the denominator of 
$\Psi(\hat h(w))$ in Eq.(\ref{nf(w)}) is equal to $|w|^{2(\ell-1)}\log|w|\psi_3(w)$ for some non-vanishing local continuous function $\psi_3$. Combining with the numerator, we conclude that if $\Psi(\hat h(w))\not\equiv 0$, then for in a neighborhood of $w=0$, we have
$$
\Psi(\hat h(w))=\left\{
\begin{matrix}[crl]
|w|^{2(j-1)}\log|w|\varphi_2(w)&\textrm{when}&|\xi_0|=1;\\
\\
\dfrac{\varphi_1(w)}{\log|w|}&\textrm{when}&|\xi_0|\neq 1,
\end{matrix}
\right.
$$
where $\varphi_1,\varphi_2$ are non-vanishing local continuous functions.
\end{proof}

The Betti multiplicities of a section are originally only defined on the points of good reduction, i.e. those points over which the fibers are regular. The proof of the previous proposition has suggested the following way to define the Betti multiplicity of a section at a point of bad reduction, in the case where the elliptic surface is given by some classifying map into a modular curve $X_{\Gamma(k)}$. 

Let $\pi:\mathcal E\rightarrow B$ be an elliptic surface given by some classifying map from $B$ to $X_{\Gamma(k)}$ and $\sigma:B\rightarrow\mathcal E$ be a non-torsion holomorphic section and $\Sigma:B\rightarrow\mathcal M_{\Gamma(k)}$ be the associated holomorphic map such that $\pi_k\circ\Sigma:B\rightarrow X_{\Gamma(k)}$ is the classifying map.
Let $w\in\Delta\subset B$ be a local coordinate chart such that $w=0$ is the only point of bad reduction in $\Delta$. Choose a coordinate chart $(\xi,\zeta)\in\mathbb C^2$ near $\Sigma(0)$ provided by the toroidal compactification as described at the beginning of this section. Write $\Sigma(w)=(\xi(w),\zeta(w))$. Recall that the singular fiber containing $\Sigma(0)$ can be defined by $\xi\zeta=0$ in the coordinate chart and we may assume that $\zeta(0)=0$. Since $\sigma$ is a section, i.e. $\pi\circ\sigma(w)=w$, it follows that $\Sigma(0)$ is a regular point on the singular fiber containing $\Sigma(0)$ and thus we have $\xi(0)\neq 0$. Moreover, as $\sigma$ is a non-torsion section, the proof of Proposition~\ref{isolated prop 2} shows that if $|\xi(0)|=1$, then $\xi(w)$ is not constant. 

\begin{definition}\label{betti at cusps}
If $|\xi(0)|=1$, then we define the Betti multiplicity $m_c$ of $\sigma$ at $c$ to be 
the vanishing order of $\xi(w)-\xi(0)$ at $w=0$. Otherwise, we define $m_c=1$.
\end{definition}

\noindent\textbf{Remarks.} $(i)$ One can check that the number $m_c$ does not depend on the choice of the integer $n$ that appears in the compactifying coordinate charts. $(ii)$ A local curve on $\mathcal M_{\Gamma(k)}$ defined by $\xi=\xi_0$ with $|\xi_0|=1$ is actually the compactification of a local curve $\gamma_{a,b}$ on $\mathcal M^0_{\Gamma(k)}$ defined by $z=a\tau+b$ for some $a,b\in\mathbb R$, where $(\tau,z)$ are the coordinates descending from $\mathcal H\times\mathbb C$. Thus, if $\Sigma$ is unramified at $w=0$, the vanishing order of $\xi(w)-\xi(0)$ at $w=0$ is just the order of contact of $\Sigma$ and the compactification $\overline{\gamma_{a,b}}$ at $c$,
which reconciles with the definition of the Betti multiplicity of $\sigma$ at a point of good reduction.


Now by Proposition~\ref{isolated prop 2} and the definition above, we have
\begin{corollary}\label{isolated cor 2}
Let $\sigma:B\rightarrow\mathcal E$ be a non-torsion section of an elliptic surface $\mathcal E$ given by some classifying map from $B$ to a modular curve $X_{\Gamma(k)}$. Let $\hat\Sigma:B\rightarrow\mathbb PT(\mathcal M_{\Gamma(k)})$ be the tautological lifting of the induced map $\Sigma:B\rightarrow\mathcal M_{\Gamma(k)}$. Then at a point of bad reduction $b\in B$ and for a coordinate chart $w\in\Delta$ near $b$ with $w(b)=0$, 
$$
\hat\Sigma^*\Psi(w)=|w|^{2(m_b-1)}(\log|w|)^{\pm 1}\phi(w),
$$
where $m_b$ is the Betti multiplicity of $\sigma$ at $b$ and $\phi(w)$ 
is a non-vanishing local continuous function.
\end{corollary}

We then have the finiteness for the set of points with higher Betti multiplicity, which has been obtained by Corvaja-Demeio-Masser-Zannier~\cite{CDMZ}.

\begin{corollary}[c.f.~\cite{CDMZ}]
Let $\sigma$ be a non-torision holomorphic section of an elliptic surface $\mathcal E\rightarrow B$, then there are finitely many points on $B$ at which the Betti multiplicity is at least 2.
\end{corollary}
\begin{proof}
When $\mathcal E$ is an elliptic surface given by a classifying map from $B$ to a modular curve $X_{\Gamma(k)}$, $k\geq 3$, the corollary follows from Corollary~\ref{isolated cor 1} and Corollary~\ref{isolated cor 2}. The general case can then be deduced by a reduction argument similar to that in Section~\ref{reduction section}.
\end{proof}

\subsection{Counting the total Betti multiplicity for elliptic surfaces with classifying maps}\label{counting section}

In what follows, we will fix an elliptic modular surface $\pi_{X}:\mathcal M:=\mathcal M_{\Gamma(k)}\rightarrow X:=X_{\Gamma(k)}$, where $k\geq 3$.

Let $\pi:\mathcal E\rightarrow B$ be an elliptic surface over a complex projective curve $B$, given by a classifying map $f:B\rightarrow X$. Let $\sigma$ be a non-torsion holomorphic section of $\mathcal E$ and $\Sigma:B\rightarrow\mathcal M$ be the induced holomorphic map such that $\pi_{X}\circ\Sigma=f$. Denote by $\hat \Sigma:B\rightarrow\mathbb PT(\mathcal M)$ the tautological lifting of $\Sigma$. 

Let $S:=f^{-1}(X\setminus X^0)\subset B$ be the set of points of bad reduction on $B$ and $R\subset B$ be the ramification locus of $f$. Let also $\mathfrak B_\sigma\subset B\setminus S$ be the finite subset of the points of good reduction consisting of points at which the Betti multiplicity of $\sigma$ is at least 2. Then $i\partial\bar\partial\log\hat \Sigma^*\Psi$ is a real-analytic $(1,1)$-form on $B\setminus (\mathfrak B_\sigma\cup R\cup S)$ by Corollary~\ref{isolated cor 1} and~\ref{isolated cor 2}. As in Section~\ref{modular section}, we are going to apply Stokes' theorem with $i\partial\bar\partial\log\hat \Sigma^*\Psi$. The limits of the boundary integrals involved are as follows.

\begin{lemma}\label{loop integral}
For every $b\in\mathfrak B_\sigma\cup R\cup S$, there exists a local coordinate chart on  $w\in\Delta$ on $B$ with $w(b)=0$, such that
$$
\lim_{\epsilon\rightarrow 0}\int_{\partial\Delta_\epsilon}\dfrac{i}{2\pi}\bar\partial\log\hat \Sigma^*\Psi
=\left\{
\begin{matrix}
m_b-r_b&\textrm{when}&b\not\in S;\\
m_b-1&\textrm{when}&b\in S,
\end{matrix}
\right.
$$
where $m_b$, $r_b$ are the Betti multiplicity and the ramification index at $b$ respectively, and $\partial\Delta_\epsilon=\{w\in\Delta:|w|=\epsilon\}$ is given the anti-clockwise orientation.
\end{lemma}
\begin{proof}
If $b\in(\mathfrak B_\sigma\cup R)\setminus S$, by Corollary~\ref{isolated cor 1}, we can choose the coordinate chart $w$ such that $\hat\Sigma^*\Psi(w)=|w|^{2(m_b-r_b)}\psi(w)$ for some non-vanishing continuous function on $\Delta$. Thus, the situation is similar to that in Eq.(\ref{loop eq}) of Section~\ref{modular section} and it suffices to check that
$$
\lim_{\epsilon\rightarrow 0}\int_{\partial\Delta_\epsilon}\bar\partial\log\psi=
\lim_{\epsilon\rightarrow 0}\int_{\partial\Delta_\epsilon}\dfrac{\bar\partial\psi}{\psi}=0.
$$
By the proof of Proposition~\ref{isolated prop},
$$
\psi(w)=\dfrac{\left|\lambda(w)-\dfrac{\overline w^m}{w^{m-1}}\overline{\rho(w)}\right|^2}{\alpha(w)}
$$
for some holomorphic function $\rho$ and real-analytic functions $\lambda$, $\alpha$ on $\Delta$ such that $\rho$, $\lambda$, $\alpha$ are non-vanishing in a neighborhood of $w=0$.
Therefore, we just need to verify that $\displaystyle\lim_{w\rightarrow 0}w\bar\partial\psi^\flat(w)=0$, where
$$
\psi^\flat(w):=\left|\lambda(w)-\dfrac{\overline w^m}{w^{m-1}}\overline{\rho(w)}\right|^2
$$
Since $\bar\partial\psi^\flat$ is easily seen to be bounded in a punctured neighborhood of $w=0$, the desired limit is indeed zero.

If $b\in S$, by Corollary~\ref{isolated cor 2}, we can choose a coordinate chart $w$ such that 
$\hat\Sigma^*\Psi(w)=|w|^{2(m_b-1)}(\log|w|)^{\pm 1}\phi(w)$ for some non-vanishing continuous function $\phi(w)$. As above, it suffices to verify that
$$
\lim_{\epsilon\rightarrow 0}\int_{\partial\Delta_\epsilon}\bar\partial\log\log|w|=
\lim_{\epsilon\rightarrow 0}\int_{\partial\Delta_\epsilon}\dfrac{\bar\partial\log|w|}{\log|w|}=0
$$
and
$$
\lim_{\epsilon\rightarrow 0}\int_{\partial\Delta_\epsilon}\bar\partial\log\phi=\lim_{\epsilon\rightarrow 0}\int_{\partial\Delta_\epsilon}\dfrac{\bar\partial\phi}{\phi}=0.
$$
The first limit is trivial. For the second one, from the proof of Proposition~\ref{isolated prop 2}, there are two possibilities for $\phi$. We will verify for one of them and the other is similar and therefore the detail will be omitted. Using the notations in the proof of Propostion~\ref{isolated prop 2}, in the case where $|\xi_0|\neq 1$, we have
$$
\phi(w)=\dfrac{\left|\nu(w)-w^jh_1(w)\log|w|\right|^2}{|h_2(w)|^2}
$$
for some integer $j\geq 1$, holomorphic functions $h_1,h_2$ and complex-valued real-analytic functions $\nu$ on $\Delta$ such that $h_1,h_2,\nu$ are non-vanishing in a neighborhood of $w=0$. Consequently, it suffices to check that $\displaystyle\lim_{w\rightarrow 0}w\bar\partial\phi^\flat(w)=0$, where
$$
\phi^\flat(w)=\left|\nu(w)-w^jh_1(w)\log|w|\right|^2.
$$
Since the worst possible singularity in $\bar\partial\phi^\flat$ is of order $\log|w|$ (when $j=1$), so the desired limit is also zero.
\end{proof}

We are now in the position to prove our integral formula for the Betti-multiplicities.

\begin{proof}[Proof of Theorem~\ref{integral formula}]
As in Section~\ref{modular section}, we will apply Stokes' theorem with $i\partial\bar\partial\log\hat \Sigma^*\Psi$ on
$\displaystyle B_\epsilon:=B\setminus\bigcup_{b\in \mathfrak B_\sigma\cup R\cup S}\overline{\Delta_\epsilon(b)}$, where $\Delta_\epsilon(b)$ is a disk of radius $\epsilon$ in a local coordinate chart $w$ near $b$ such that $w(b)=0$. By Lemma~\ref{loop integral},
\begin{eqnarray*}
-\lim_{\epsilon\rightarrow 0}\int_{B_\epsilon}\dfrac{i}{2\pi}\partial\bar\partial\log\hat \Sigma^*\Psi&=&\sum_{b\in(\mathfrak B_\sigma\cup R)\setminus S}(m_b-r_b)+\sum_{b\in S}(m_b-1)\\
&=&\sum_{b\in\mathfrak B_\sigma}(m_b-r_b)+\sum_{b\in R\setminus (S\cup\mathfrak B_\sigma)}(1-r_b)+\sum_{b\in S}(m_b-1)\\
&=&\sum_{b\in\mathfrak B_\sigma}(m_b-1)+\sum_{b\in R\setminus S}(1-r_b)+\sum_{b\in S}(m_b-1)\\
&=&\sum_{b\in B}(m_b-1)-\sum_{b\in B\setminus S}(r_b-1).
\end{eqnarray*}

On the other hand, since the restriction of the classifying map $f:B_\epsilon\rightarrow f(B_\epsilon)\subset X^0$ is an unramified cover, by pulling back the bundles, metrics and connections, etc. from $X^0$ to $B_\epsilon$, we deduce from Proposition~\ref{equivalence} and Eq.(\ref{iddbarlogeta}), Eq.(\ref{integrable prop}) of Section~\ref{modular section} that
$$
\lim_{\epsilon\rightarrow 0}\int_{B_\epsilon}\dfrac{i}{2\pi}\partial\bar\partial\log\hat\Sigma^*\psi=-\lim_{\epsilon\rightarrow 0}\int_{B_\epsilon}\dfrac{1}{2\pi}f^*\omega=-\int_{B\setminus S}\dfrac{1}{2\pi}f^*\omega
=-\dfrac{d}{2\pi}\int_{X^0}\omega,
$$
where $d=\deg(f)$. The proof is now complete.
\end{proof}

\begin{proof}[Proof of Corollary~\ref{inequality cor}]
Since $m_b\geq 1$ for every $b\in B$, by Theorem~\ref{integral formula} and Eq.(\ref{area formula}) in the Introduction, we have
\begin{eqnarray*}
|\mathfrak B_\sigma|&\leq&\sum_{b\in B\setminus S}(r_b-1)+d(g(X)-1)+\dfrac{d\nu_\infty(X)}{2}\\
&=&\sum_{b\in B}(r_b-1)+d(2g(X)-2)+\dfrac{d\nu_\infty(X)}{2}-\sum_{b\in S}(r_b-1)-d(g(X)-1)\\
&=&2g-2+d\nu_\infty(X)-\sum_{b\in S}(r_b-1)-d(g(X)-1)-\dfrac{d\nu_\infty(X)}{2}\\
&=&2g-2+|S|-d\left(g(X)-1+\dfrac{\nu_\infty(X)}{2}\right)\\
&=&2g-2-\deg(f^*(K_{X}\otimes S_{X})^{\frac{1}{2}}))+|S|
\end{eqnarray*}
\end{proof}

In~\cite{UU}, Ulmer-Urz\'ua obtained the following inequality (after incorporating some of our notations)
$$
|\mathfrak B_\sigma|\leq 2g-2-\deg(O^*(\Omega^1_{\mathcal E/B}))+|S|,
$$
where $O$ is the zero section of $\mathcal E$. This inequality is in fact equivalent to Corollary~\ref{inequality cor}, as we have
\begin{proposition}\label{line bundles}
$f^*(K_{X}\otimes S_{X})^{\frac{1}{2}}=O^*(\Omega^1_{\mathcal E/B})$
\end{proposition}
\begin{proof}
Since $\mathcal E$ is obtained by the classifying map $f:B\rightarrow X$ and $\mathcal E,\mathcal M$ are relatively minimal, it follows that $O^*(\Omega^1_{\mathcal E/B})=f^*(O^*(\Omega^1_{\mathcal M/X}))$, where we use the same symbol $O$ to denote the zero section of $\mathcal M$. Hence, it suffices to prove that 
$$
(K_{X}\otimes S_{X})^{\frac{1}{2}}=O^*(\Omega^1_{\mathcal M/X}).
$$

Let $c_\infty\in X$ be the cusp corresponding to $i\infty$. Denote a point on $\mathcal H\times\mathbb C$ by $(\tau, z)$. Recall that there exists a positive number $M$ such that for any $\tau,\tau'\in\mathcal H$ with $\im\tau>M$ and $\im\tau'>M$, the two points $\tau$ and $\tau'$ can be $\mathrm{SL}(2,\mathbb Z)$-equivalent only if $\tau'=\tau+m$ for some $m\in\mathbb Z$. Consequently, there exists some neighborhood $\mathcal U\subset X$ containing $O(c_\infty)$ such that the vector field $\dfrac{\partial}{\partial z}$ on $\mathcal H\times\mathbb C$ descend to $\mathcal U\cap X^0$. Now using the explicit coordinate charts on the toroidal compactification $X$ given in Section~\ref{toroidal}, it can be easily checked that $\dfrac{\partial}{\partial z}=2\pi i\left(-\xi\dfrac{\partial}{\partial\xi}+\zeta\dfrac{\partial}{\partial\zeta}\right)$  and so it extends to a non-vanishing section of $\Omega^{-1}_{\mathcal M/X}$ if $\mathcal U$ is sufficiently small. Thus, we get a non-vanishing section $O^*\left(\dfrac{\partial}{\partial z}\right)$ of $O^*(\Omega^{-1}_{\mathcal M/X})$ in a neighborhood of $c_\infty$ on $X$.  Moreover, by how $\mathrm{SL}(2,\mathbb Z)\ltimes\mathbb Z^2$ acts on $\mathcal H\times\mathbb C$, we deduce that once we have fixed an isomorphism $O^*(\Omega^{-2}_{\mathcal M^0/X^0})\cong T(X^0)$ (c.f. Section~\ref{eta section}), we have $O^*\left(\dfrac{\partial}{\partial z}\otimes \dfrac{\partial}{\partial z}\right)=a\dfrac{\partial}{\partial\tau}$ as local sections near $c_\infty$, where $a\in\mathbb C$ is a non-zero constant. Now in terms of the compactifying coordinate $q=e^{2\pi i\tau/k}$ on a neighborhood $U$ of $c_\infty\in X$, we have $\dfrac{\partial}{\partial\tau}=\dfrac{2\pi i}{k}q\dfrac{\partial}{\partial q}$ and so we get a sheaf isomorphism from $O^*(\Omega^{-2}_{\mathcal M/X})|_U$ to $(T(X)\otimes (-S_X))|_U$. 

Finally, if $c\in X$ is another cusp, we can conjugate everything by some $\gamma\in \mathrm{SL}(2,\mathbb Z)$ and repeat the previous argument. We have thereby shown that by fixing an isomorphism $O^*(\Omega^{-2}_{\mathcal M^0/X^0})\cong T(X^0)$, we get a sheaf isomorphism $O^*(\Omega^{-2}_{\mathcal M/X})\cong T(X)\otimes(-S_X)$ and the desired result follows.
\end{proof}

\subsection{The case for general non-isotrivial elliptic surfaces}\label{reduction section}

We will now consider a non-isotrivial elliptic surface $\pi:\mathcal E\rightarrow B$. 
In order to reduce the case with classifying maps, we first show that there exists a finite branched cover $\nu:B'\rightarrow B$, which is unbranched outside the points of bad reduction $S\subset B$, such that the pullback of $\mathcal E$ on $B'$ is birational to an elliptic surface given by a classifying map. 

In what follows, we will fix an integer $k\geq 3$, write $X = X_{\Gamma(k)}$ and $\pi_X:\mathcal M \to X$ for the elliptic modular surface with level-$k$ structure.

\begin{proposition}\label{lifting}
Let $\pi: \mathcal E \to B$ be a non-isotrivial elliptic surface, $B^0 \subset B$ be the dense Zariski open subset corresponding to regular fibers. Then, there exists an unramified finite cover $\nu_0: B^0_\sharp \to B^0$, which extends to a morphism $\nu: B_\sharp \to B$ for a complex projective curve $B_\sharp \supset B^0_\sharp$ such that $B^0_\sharp$ is Zariski open in $B_\sharp$, and such that there exists a classifying map $f_0: B^0_\sharp \to X^0$, which extends to a morphism $f: B_\sharp \to X$.
\end{proposition}
\begin{proof}
The elliptic surface $\mathcal E$ induces a representation $\rho$ of the fundamental group of $B^0$ in $\mathrm{SL}(2,\mathbb Z)$. Let $\Gamma_\sharp:=\rho(\pi_1(B^0))\cap\Gamma(k)\subset\mathrm{SL}(2,\mathbb Z)$. Since $\Gamma(k)$ is of finite index in $\mathrm{SL}(2,\mathbb Z)$, so is $\rho^{-1}(\Gamma_\sharp)\subset\pi_1(B^0)$. Thus, this gives a finite unramified cover $\nu_0:B^0_\sharp\rightarrow B^0$. It follows that $B^0_\sharp$ naturally admits the structure of a Riemann surface of finite type, and we denote by $B^0_\sharp \subset B_\sharp$ its compactification to a compact Riemann surface $B_\sharp$, from which automatically we have an extension of $\nu_0: B^0_\sharp \to B^0$ to a holomorphic map $\nu: B_\sharp \to B$.  

By construction we have a holomorphic classifying map $f_0: B^0_\sharp \to X^0$ such that $f_0^*\mathcal M$ is isomorphic to $\nu_0^*\mathcal E$. To finish the proof of the proposition, it remains to check that $f_0: B^0_\sharp \to X^0$ extends to a holomorphic map $f: B_\sharp \to X$. This follows from an extension theorem of Borel~\cite{Bo}, which states that any holomorphic map from a punctured disk $ \Delta^*$ to $X^0$ extends to a holomorphic map from $\Delta$ to $X$. We thereby obtain the desired holomorphic extension $f:B_\sharp\rightarrow X$. 
\end{proof}

Let $\pi:\mathcal E\rightarrow B$ a relatively minimal non-isotrivial elliptic surface over a complex projective curve $B$. Let $\nu:B_\sharp\rightarrow B$ be the finite branched cover and $f:B_\sharp\rightarrow X$ be the morphism into the modular curve $X$ which extends a classifying map $f_0:B^0_\sharp\rightarrow X^0$, given by Proposition~\ref{lifting}. By our construction of $f_0$ and $f$, there is a relatively minimal elliptic surface $\pi_\sharp: \mathcal E_\sharp \to B_\sharp$ which is birational to $\nu^*\mathcal E$ and also birational to $f^*\mathcal M$ $($while $\mathcal E^0_\sharp := \mathcal E_\sharp|_{B^0_\sharp}$ is isomorphic to $\nu_0^*\mathcal E$ and $f_0^*\mathcal M$ over $B^0_\sharp)$.

For any non-torsion section $\sigma$ of $\mathcal E$, by pulling back with $\nu^*$, we obtain a non-torsion section $\sigma_\sharp$ of $\mathcal E_\sharp$. Since $\mathcal E_\sharp$ has a classifying map, by Corollary~\ref{inequality cor} and Proposition~\ref{line bundles},
$$
|\mathfrak B_{\sigma_\sharp}|\leq 2g_\sharp-2-\deg(O_\sharp^*(\Omega^1_{\mathcal E_\sharp/B_\sharp}))+|S_\sharp|,
$$
where $g_\sharp$ is the genus of $B_\sharp$ and $S_\sharp:=B_\sharp\setminus B^0_\sharp$. As both $\mathcal E$ and $\mathcal E_\sharp$ are relatively minimal, there exist neighborhoods $U$, $U_\sharp$ of the images of zero sections $\mathcal O\subset\mathcal E$ and $\mathcal O_\sharp\subset\mathcal E_\sharp$ respectively, such that there is a fiber-preserving biholomorphism $U_\sharp\cong\nu^*U$. From this we deduce that $\nu^*(O^*(\Omega^1_{\mathcal E/B}))\cong O_\sharp^*(\Omega^1_{\mathcal E_\sharp/B_\sharp})$.
Let $g$ be the genus of $B$ and $n$ be the degree of $\nu:B_\sharp\rightarrow B$.
Recall that $\nu|_{B^0_\sharp}$ is unramified and let $r_{b_\sharp}$ be the ramification index of $\nu$ at a point $b_\sharp\in S_\sharp$. Then, 
\begin{eqnarray*}
|\mathfrak B_{\sigma}|&=&\dfrac{|\mathfrak B_{\sigma_\sharp}|}{n}\\
&\leq&\dfrac{1}{n}\left( 2g_\sharp-2-n\deg(O^*(\Omega^1_{\mathcal E/B}))+|S_\sharp|\right)\\
&=&2g-2+\dfrac{\sum_{b_\sharp\in S_\sharp}(r_{b_\sharp}-1)}{n}-\deg(O^*(\Omega^1_{\mathcal E/B}))+\dfrac{|S_\sharp|}{n}
\\
&=&2g-2-\deg(O^*(\Omega^1_{\mathcal E/B}))+|S|,
\end{eqnarray*}
where $S=\nu(S_\sharp)=B\setminus B^0$.

\vskip 0.4cm
\noindent
{\bf Acknowledgments.} \quad  For the research undertaken in the current article the first author was supported by GRF grant 17301518 of the HKRGC and the second author was partially supported by Science and Technology Commission of Shanghai Municipality (STCSM) (No. 13dz2260400). The first author's interests in applying complex geometry to arithmetic problems were very much rekindled upon exchanges with the late Professor Nessim Sibony on the interactions between complex geometry and number theory during and after a research visit to l'Universit\'e de Paris (Orsay) at the turn of the millennium. The second author has been lucky enough to have had many inspiring chats with the amiable Professor Sibony on numerous occasions in China and France.  Both authors would like to dedicate this article to the memory of Professor Sibony as a distinguished teacher and researcher, and as a leader in the mathematical community.

\end{document}